\documentclass[12pt,  reqno]{amsart}
\pdfoutput=1
\usepackage[margin=1.2in,marginparwidth=1.5cm, marginparsep=0.5cm]{geometry}

\numberwithin{equation}{section}
\usepackage{mathrsfs}
\usepackage{amsfonts}
\usepackage{amsmath}
\usepackage{stmaryrd}
\usepackage{amssymb}
\usepackage{amsthm}
\usepackage{mathrsfs}
\usepackage{url}
\usepackage{amsfonts}
\usepackage{amscd}
\usepackage{indentfirst}
\usepackage{enumerate}
\usepackage{amsmath,amsfonts,amssymb,amsthm}
\usepackage{amsmath,amssymb,amsthm,amscd}
\usepackage{graphicx,mathrsfs}
\usepackage{appendix}
\usepackage{ulem}

\usepackage[numbers,sort&compress]{natbib}

\usepackage{color}
\usepackage[colorlinks, linkcolor=blue, citecolor=blue]{hyperref}

\newcommand{\R}{\mathbb{R}}

\newtheorem{Thm}{Theorem}[section]
\newtheorem{Lem}[Thm]{Lemma}
\newtheorem{Prop}[Thm]{Proposition}

\newtheorem{Def}[Thm]{Definition}
\newtheorem{Rem}[Thm]{Remark}

\newtheorem{theo}{Theorem}
\newtheorem{corollary}{Corollary}

\begin{document}

\title
[NLS system with three waves interaction]
{Standing waves for a Schr\"odinger system with three waves interaction}

\author[L. Forcella]{Luigi Forcella}
\address{
\newline\indent
Luigi Forcella, Dipartimento di Matematica, Universit\`a Degli Studi di Pisa, 
\newline\indent  Largo Bruno Pontecorvo, 5, 56127, Pisa, Italy}
\email{luigi.forcella@unipi.it}

\author[X. Luo]{Xiao Luo}
\address{
\newline\indent Xiao Luo, School of Mathematics, Hefei University of Technology,
\newline\indent Hefei, 230009, P. R. China}
\email{luoxiao@hfut.edu.cn}

\author[T. Yang]{Tao Yang}
\address{
\newline\indent Tao Yang, Department of Mathematics, Zhejiang Normal University,
\newline\indent Jinhua, 321004, P. R. China}
\email{yangtao@zjnu.edu.cn}

\author[X. Yang]{Xiaolong Yang}
\address{
\newline\indent School of Mathematics and Statistics, Central China Normal University,
\newline\indent Wuhan, 430079, P. R. China}
\email{yangxiaolong@mails.ccnu.edu.cn}

\begin{abstract}
We study standing waves for a system of nonlinear Schr\"odinger equations with three waves interaction arising as a model for the Raman amplification in a plasma. We consider the mass-critical and mass-supercritical regimes, and we prove existence of ground states along with a synchronized mass collapse behavior. In addition, we show that the set of ground states is stable under the associated Cauchy flow. Furthermore, in the mass-supercritical setting we construct an excited state that corresponds to a strongly unstable standing wave. Moreover, a semi-trivial limiting behavior of the excited state is drawn accurately. Finally, by a refined control of the excited state's energy, we give sufficient conditions to prove global existence or blow-up of solutions to the corresponding Cauchy problem.
\end{abstract}

\maketitle

{\small\small
\keywords {\noindent {\bf Keywords:}
 {NLS system; standing waves; stability.}
\smallskip
\newline
\subjclass{\noindent {\bf 2010 Mathematics Subject Classification:} 35Q55, 35A15, 35B35.

\section{Introduction}

In this paper, we consider a three-components system of nonlinear Schr\"{o}dinger equations related to the Raman amplification in a plasma, as derived by Colin, Colin, and Ohta in \cite{CCO},  which reads as follows:
\begin{equation}\label{eqA0.1}
\begin{cases}
i \partial_{t} \psi_{1}=-\Delta \psi_{1}-\left|\psi_{1}\right|^{p-2} \psi_{1}-\alpha \psi_{3} \overline{\psi}_{2}, \\
i \partial_{t} \psi_{2}=-\Delta \psi_{2}-\left|\psi_{2}\right|^{p-2} \psi_{2}-\alpha \psi_{3} \overline{\psi}_{1}, \\
i \partial_{t} \psi_{3}=-\Delta \psi_{3}-\left|\psi_{3}\right|^{p-2} \psi_{3}-\alpha \psi_{1} \psi_{2}.
\end{cases}
\end{equation}
Here, $\psi_j=\psi_j(t,x)$ with $j=1,2,3$, are complex-valued functions $\psi_j:\R\times \R^N\mapsto \mathbb C$, with $\overline \psi_j$ denoting the complex conjugate,  the space dimension is $N\le 3$,  $\alpha$ is a positive real parameter,
and the power non-linearity $p$ is in the range $2_*\le p< 2^*$, where
\begin{equation*}
\begin{cases}
2_*=2+\frac{4}{N},\\
 2^*=\infty \hbox{\, if \,} N\leq2,\quad  2^*=\frac{2N}{N-2}  \hbox{\, if \,} N=3.
\end{cases}
\end{equation*}
Namely, we consider the mass-critical or mass-supercritical  and energy-subcritical power-type non-linearities. \medskip

It is standard to see that the Cauchy problem associated to \eqref{eqA0.1} is locally well-posed in the energy space, i.e., for a fixed initial datum
\[
(\psi_{0,1}, \psi_{0,2}, \psi_{0,3})(x):=(\psi_1, \psi_2, \psi_3)(0,x)\in H^1(\R^N,\mathbb{C}^3),
\] there exists a unique solution $(\psi_1, \psi_2, \psi_3)\in C([0, T_{\max}),H^1(\R^N,\mathbb{C}^3) ) $, where $T_{\max}>0$ is the positive maximal time of existence (a similar notion can be given for negative times). See the monograph \cite{CT}. Moreover, the blow-up alternative holds true, in the sense that either $T_{\max}=\infty$ (the solution is global), or $T_{\max}<\infty$  and the homogeneous Sobolev norm of the solution diverges as $t\to T_{\max}^-$. More  precisely, if $T_{\max}<\infty$, then  $\lim_{t\to T_{\max}^-}\left(\sum_{i=1}^{3}\|\nabla \psi_i(t)\|^2_{L^2 (\R^N)} \right)=\infty$.

\noindent In addition, the following quantities are conserved along the flow: the energy, defined by
\begin{equation}\label{eq:energy}
E(t)=E(\vec{\psi}(t))=\sum^{3}_{i=1}\left(\frac{1}{2}\|\nabla \psi_{i}(t)\|^2_{L^2(\R^N)}-\frac{1}{p}\|\psi_i(t)\|^p_{L^p(\R^N)}  \right)-\alpha \mathrm{Re}\int_{\R^N} \left(\psi_1\psi_2\overline{\psi}_3\right)(t)dx,
\end{equation}
and the mixed masses
\begin{equation}\label{eq:cons-masses}
\begin{aligned}
Q_1(t)&=Q_1(\vec{\psi}(t))=\|\psi_1(t)\|^2_{L^2(\R^N)}+\|\psi_3(t)\|^2_{L^2(\R^N)} \\  Q_2(t)&=Q_2(\vec{\psi}(t))=\|\psi_2(t)\|^2_{L^2(\R^N)}+\|\psi_3(t)\|^2_{L^2(\R^N)},
\end{aligned}
\end{equation}
where we used the compact notation \[
\vec{\psi}=\vec{\psi}(t,x)=\left(\psi_1(t,x),\psi_2(t,x),\psi_3(t,x)\right)\in H^1(\R^N,\mathbb{C}^3).
\] As usual, conservation means that the previous quantities are not dependent on time, or alternatively $E(t)=E(0)$,  $Q_1(t)=Q_1(0)$, and $Q_2(t)=Q_2(0)$ for any time $t$ in the maximal interval of existence $[0,T_{\max})$. The conservation laws can be showed by a standard regularization argument, see \cite{CT}.
\\
\noindent Furthermore, we note that   \eqref{eqA0.1} can be written as
\begin{equation*}
\partial_t \vec{\psi}(t,x)=-i E'(\vec{\psi}(t,x)),
\end{equation*}
and that
\begin{equation*}
E(e^{i\theta_1}u_1, e^{i\theta_2}u_2, e^{i(\theta_1+\theta_2)}u_3)=E(\vec{u}), \end{equation*}
for any $(\theta_1,\theta_2)\in \R^2$, and any function $\vec u=(u_1,u_2,u_3) \in H^1(\R^N,\mathbb{C}^3)$. \medskip

The main purpose of this paper is to study existence and stability properties of \textit{standing waves solutions} to \eqref{eqA0.1}. Let us recall that a standing wave for \eqref{eqA0.1} is a solution of the form $\left(\psi_1(t,x),\psi_2(t,x),\psi_3(t,x)\right)$ with $\psi_1(t,x)=e^{i\lambda_1t}u_1(x)$, $\psi_2(t,x)=e^{i\lambda_2t}u_2(x)$ and $\psi_3(t,x)=e^{i\lambda_3t}u_3(x)$, where $\lambda_1,\lambda_2,\lambda_3$ are real numbers and $\vec{u}\in H^1(\R^N,\mathbb{C}^3)$ satisfies the system of elliptic equations
\begin{equation}\label{eqA0.2}
\begin{cases}
-\Delta u_{1}+ \lambda_1u_{1}=\left|u_{1}\right|^{p-2} u_{1}+\alpha u_{3} \overline{u}_{2}, \\
-\Delta u_{2}+ \lambda_2u_{2}=\left|u_{2}\right|^{p-2} u_{2}+\alpha u_{3} \overline{u}_{1}, \\
-\Delta u_{3}+ \lambda_3u_{3}=\left|u_{3}\right|^{p-2} u_{3}+\alpha u_{1} u_{2},
\end{cases}
\end{equation}
where $\lambda_3=\lambda_1+\lambda_2$.\medskip

Under certain conditions, the existence, uniqueness and multiplicity of solutions of \eqref{eqA0.2} have been studied by many authors. We refer the reader to \cite{LO,CZT,PA,WJ} and the references therein.  In particular, the authors of \cite{CCO,CCO1} studied the orbital stability of  solutions (semi-trivial standing waves) for system \eqref{eqA0.1} of the form $(e^{i\omega t}u,0,0)$,  $(0,e^{i\omega t}u,0)$, $(0,0,e^{i\omega t}u)$, (such kind of solutions, with two trivial components, are called \textit{scalar solutions}) where $\omega>0$ and $u\in H^1(\R^N,\R)$ is the unique positive radial solution of
\begin{equation*}
-\Delta u+\omega u=|u|^{p-2}u \quad \text{in}\quad  \R^N.
\end{equation*}
In \cite{CC, CCO}, it is proved that when $2<p<2_*$, $(e^{i\omega t}u,0,0)$ and $(0,e^{i\omega t}u,0)$ are orbitally stable for any $\alpha>0$, while $(0,0,e^{i\omega t}u)$ is orbitally stable if $0<\alpha<\bar \alpha$, and it is orbitally unstable if $\alpha>\bar \alpha$ for a suitable positive constant $\bar \alpha=\bar\alpha(N,p,\omega)$ (see also \cite{MM} in the higher dimensions $N = 4, 5$).  \\
In \cite{AH}, it is instead proved the existence of stable standing waves (vector solutions) for the system \eqref{eqA0.1} with $N=1$, $2<p<6=2_*$ (i.e., the mass-subcritical case) and $\alpha>0$, by minimizing the energy $E(\vec{u})$ on the manifold
\begin{equation}\label{def:S12}
S(a_1,a_2):=\left\{\vec{u}\in H^1(\R^N,\mathbb{C}^3) \hbox{ s.t. } \int_{\R^N} |u_1|^2+|u_3|^2 dx=a^2_1,\quad \int_{\R^N} |u_2|^2+|u_3|^2 dx=a^2_2\right\},
\end{equation}
where $a_1,a_2>0$.
The results of \cite{AH} have been generalized in \cite{KO} to the higher dimensional case and to the model \eqref{eqA0.1} with potentials (see also \cite{OY}). It is worth mentioning that in \cite{PA}, the existence of non-scalar solutions were proved by minimizing the action function on the Nehari manifold, provided the coupling parameter $\alpha$ is large enough. \medskip

In this paper, illuminated by \cite{JL} and \cite{Soave1}, we aim to consider standing waves and their stability for the system \eqref{eqA0.1} in the  mass-critical or mass-supercritical regime and the energy-subcritical ones, namely we cover the range  of non-linearities $2_*\leq p<2^*$, where the corresponding energy functional $E(\vec{u})$ is not always bounded from below on $S(a_1,a_2)$. Note that the  coupling terms are of mass-subcritical type and sign-indefinite, then we are dealing with a special mass-mixed case (i.e., the combination of mass-subcritical and mass-supercritical terms), which is more complicated. \medskip

Before introducing the main results, we recall some definition (see also \cite{BeJe}).
\begin{Def}
We say that $\vec{u}_0$ is a ground state of \eqref{eqA0.2} on $S(a_1,a_2)$ if
\begin{equation*}
dE \vert_{S(a_1,a_2)}(\vec{u}_0)=0 \quad \hbox { and } \quad E(\vec{u}_0)=\inf \left\{E(\vec{u}) \hbox{ s.t. } dE\vert_{S(a_1,a_2)}(u)=0\ \hbox{and}\ \vec{u}\in S(a_1,a_2)\right\}.
\end{equation*}
We say that $\vec{v}_0$ is an excited state of \eqref{eqA0.2} on $S(a_1,a_2)$ if
\begin{equation*}
dE\vert_{S(a_1,a_2)}(\vec{v}_0)=0 \quad \text { and } \quad E(\vec{v}_0)>\inf \left\{E(\vec{u}) \hbox{ s.t. } dE\vert_{S(a_1,a_2)}(u)=0\ \text{and}\ \vec{u}\in S(a_1,a_2)\right\}.
\end{equation*}
The set of ground states will be denoted by $\mathcal G=\mathcal G_{p,\alpha,N}$.
\end{Def}
We emphasize, as in \cite{AH}, that  variational problems with the energy restricted on the manifold $S(a_1,a_2)$ is particularly appropriate for the study of the stability properties of the ground states, as both the energy and the partial mass functionals $Q_1$ and $Q_2$ are conserved along the flow generated by \eqref{eqA0.1}.

\begin{Def} \rm
\textup{(i)} We say that the set $\mathcal G$ is orbitally stable if $\mathcal G\neq \emptyset$ and for any $\varepsilon>0$, there exists a $\delta>0$ such that,  provided that an initial datum $\vec{\psi}(0)=\left(\psi_1(0),\psi_2(0),\psi_3(0)\right)$  satisfies
\begin{equation*}
\inf_{\vec{u}\in \mathcal G}\|\vec{\psi}(0)-\vec{u}\|_{H^1(\R^N, \mathbb C^3)}< \delta,
\end{equation*}
then the corresponding solution $\vec{\psi}$ to \eqref{eqA0.1} is globally defined and
\begin{equation*}
\inf_{\vec{u}\in \mathcal G}\|\vec{\psi}(t)-\vec{u}\|_{H^1(\R^N, \mathbb C^3)}<\varepsilon \quad\quad \forall t\in\R.
\end{equation*}
\smallskip

\noindent \textup{(ii)} A standing wave $(e^{i\lambda_1t}u_1,e^{i\lambda_2t}u_2,e^{i\lambda_3t}u_3)$ is said to be strongly unstable if for any $\varepsilon>0$ there exists $\vec{\psi}_0\in H^1(\R^N,\mathbb{C}^3)$ such that $\|\vec{u}-\vec{\psi}_0\|_{H^1(\R^N, \mathbb C^3)}< \varepsilon$, and $\vec{\psi}(t)$ blows-up in finite time, namely $T_{\max}<\infty$.
\end{Def}

Throughout this article, we are not only interested in proving existence of standing waves and their stability properties, but also in proving suitable asymptotic results for different regimes depending on the involved parameters $\alpha$, $a_1,$ and $a_2$. To this aim, before stating our first main result, we introduce another minimization problem:
\begin{equation}\label{def:m0}
m_0(a_1,a_2):=\inf_{\vec{u}\in S(a_1,a_2)}E_0(\vec{u}),
\end{equation}
where
\begin{equation*}
E_0(\vec{u}):=\frac{1}{2}\sum^{3}_{i=1}\|\nabla u_i\|^2_{L^2(\R^N)}-\mathrm{Re}\int_{\R^N} u_1u_2\overline{u}_3dx.
\end{equation*}

We can now state our main result regarding existence, stability, and mass-synchronised  asymptotic of the ground states.
\begin{theo}\label{th1.1}
Let $N\le 3$,  $2_*\le p<2^*$, and $\alpha,a_1, a_2>0$. There exists a positive explicit constant $D = D(N, p,\alpha)$ such that if $\max\{a_1,a_2\}<D$, we have:\smallskip

\noindent \textup{(i)} $\mathcal G$ is nonempty, i.e., there exists a ground state of \eqref{eqA0.2} on $S(a_1,a_2)$;\smallskip

\noindent\textup{(ii)} the set $\mathcal G$ is orbitally stable;\smallskip

\noindent\textup{(iii)} fix $\alpha > 0$ and let $\vec{u}\in \mathcal G$. Assume that $a_2=a_1\to 0$, then we have
\begin{equation*}
\sup_{\vec{u}\in \mathcal G}\|\vec{u}(x)-\kappa\alpha^{-1}\vec{v}_0(\kappa^{\frac{1}{2}}x)\|_{H^1(\R^N,\mathbb{C}^3)}= o(1),\end{equation*}
where $\vec{v}_0$ is a minimizer for $m_0(\sqrt{2}\|w\|_{L^2(\R^N)},\sqrt{2}\|w\|_{L^2(\R^N)})$ (see definition \eqref{def:m0}), the scaling constant $\kappa=\left(\frac{\alpha a_1}{\sqrt{2}\|w\|_{L^2(\R^N)}} \right)^{\frac{4}{4-N}}$, and $w$ is the unique, real   positive solution of $-\Delta w+w=w^2$;\smallskip

\noindent\textup{(iv)} if $\vec{u}\in \mathcal G$ then $\sum^{3}_{i=1}\|\nabla u_i\|^2_{L^2(\R^N)}\to 0$ as $\alpha \to 0$.
\end{theo}
We comment on the results given in   Theorem  \ref{th1.1} above.
\begin{Rem}\rm
To the best of the authors' knowledge, this is the first result dealing with the existence and stability/instability results of standing waves for the Schr\"odinger system with three waves interaction in  the mass-critical/supercritical non-linearities. Moreover, it is worth mentioning that our result are not perturbative, indeed the constant $D$ in the statement of Theorem \ref{th1.1} is given by
\begin{equation}\label{def:D}
D:=\left(\frac{3(p\gamma_p-2)}{\alpha (2p\gamma_p-N) C^3(N,p)}\right)^{\frac{N(p-2)-4}{4(p-3)}}
\left(\frac{p(4-N)}{2(2p\gamma_p-N)C^p(N,p)}\right)^{\frac{4-N}{4(p-3)}},
\end{equation}
where $C(N,p)$ is the best constant in the following Gagliardo-Nirenberg inequality,
\begin{equation}\label{a1}
 \|u\|_{L^p(\R^N)} \leq C(N,p) \|\nabla u\|_{L^2(\R^N)}^{\gamma_p} \|u\|_{L^2(\R^N)}^{1-\gamma_p}, \quad \forall  u \in {H}^{1}(\R^N,\mathbb{C}),
\end{equation}
with
\begin{equation}\label{def:gammap}
\gamma_p=\frac{N(p-2)}{2p}, \quad p\in [2,2^*).
\end{equation}
\end{Rem}

\begin{Rem}\rm
Theorem \ref{th1.1} shows that a ground state exists even if $E|_{S(a_1,a_2)}$ is unbounded from below, and, for $a_1,a_2$ small enough, the ground state is indeed a least action solution which reaches the infimum of the $C^1$ action functional $\displaystyle J(\vec{u})=E(\vec{u})+\frac{1}{2}\sum^{3}_{i=1}\lambda_i\|u_{i}\|^2_{L^2(\R^N)}$ among all nontrivial solutions to \eqref{eqA0.2} (see \cite{PA,WJ} for the existence of least action solutions), where ${\lambda_i}$ ($i=1,2,3$) are the Lagrange multipliers corresponding to the ground state. 
\end{Rem}
\begin{Rem}\rm
The set $\mathcal G$, containing a priori complex-valued ground states, has the following structure:
\begin{equation*}
\mathcal G=\left\{ (e^{i\theta_1}u_1,e^{i\theta_2}u_2,e^{i(\theta_1+\theta_2)}u_3) \quad \hbox{s.t.} \quad \theta_1,\theta_2\in \R \right\},
\end{equation*}
where $(u_1,u_2,u_3)\in S(a_1,a_2)$ is a positive, radial ground state of \eqref{eqA0.2}. See the proof of Theorem \ref{th1.1} later on. Since now on, we refer to a radial $\vec u$ is each component is radial.
\end{Rem}
\begin{Rem}\rm
The fact that $\mathcal G$ is orbitally stable indicates that the  coupling term leads to the stabilization of the  standing waves corresponding to \eqref{eqA0.1}. It is worth recalling that for the Schr\"odinger  equation $i \partial_{t} \psi=-\Delta \psi-\left|\psi\right|^{p-2} \psi$, for $p$ in the mass-supercritical regime, the standing wave $\psi(t,x)=e^{i\lambda t}u(x)$ is  strongly unstable, see \cite{CT}, where $u\in H^1(\R^N)$ is the unique positive radial solution of $-\Delta u+ \lambda u=\left|u\right|^{p-2} u$ for $\lambda>0$.
\end{Rem}
\begin{Rem}\rm
In proving the existence of ground states, due to the indefinite sign of the three wave interaction term in the  energy functional, we need to introduce an additional constrain given by an inequality. This in turn makes appear further difficulties in proving the compactness of related minimizing sequences, and is different from constrained variational problems with a sign-definite type structure, see for example  \cite{BJS,JL,MS21,Soave1,WW}.
In order to get the synchronized mass collapse behavior of the ground state of \eqref{eqA0.2} on $S(a_1,a_2)$ (namely, the claim of point \textup{(iii)} in Theorem \ref{th1.1}), we  prove the existence of ground states for the limit system
\begin{equation}\label{n1}
\begin{cases}
-\Delta u_{1}+\lambda_1u_{1}=u_{3}\overline{u}_{2},\\
-\Delta u_{2}+\lambda_2u_{2}=u_{3}\overline{u}_{1},\\
-\Delta u_{3}+(\lambda_1+\lambda_2)u_{3}=u_{1}u_{2},
\end{cases}
\end{equation}
under the constraints
\begin{equation}\label{eq:constr}
Q_1(\vec{u})=a^2_1 \quad \hbox{ and } \quad Q_2(\vec{u})=a^2_2.
\end{equation}
If $\lambda_1=\lambda_2$, the
uniqueness of minimizer for $m_0(a_1,a_2)$ (see \eqref{def:m0}) and ground state for \eqref{n1} are proved in \cite{WJ,LO}. Moreover, for $N=3$, by replacing the constraints in \eqref{eq:constr} by three independent prescribed mass constraints, a modification of the proof of Lemma  \ref{lem4.2} gives a positive answer to the open problem proposed by Kurata and Osada in \cite[Remark 4]{KO}. See Remark \ref{rem:KO-sol}.
\end{Rem}

We now give the results related to the existence and properties of excited states. In what follows, we consider mass-energy intercritical non-linearities, namely $2_*<p<2^*$.
\begin{theo}\label{th1.2}
Let $N\le 3$,  $2_*<p<2^*$, $\max\{a_1,a_2\}<D$, and $a_1, a_2>0$. There exists $\alpha_0=\alpha_0(a_1,a_2)>0$ such that, for any $\alpha>\alpha_0$:\smallskip

\noindent \textup{(i)} there exists an excited state $\vec{v}=(v_1,v_2,v_3)\in S(a_1,a_2)$, with associated Lagrange multipliers $\lambda_1,\lambda_2>0$;\smallskip

\noindent \textup{(ii)} let $a_1> 0$ and  $a_2\to 0^+$, then
we have
\begin{equation*}
\left(\tilde{\kappa}^{-\frac{1}{p-2}}v_1(\tilde{\kappa}^{-\frac{1}{2}}x),v_2(x),v_3(x)\right)\to (w_p,0, 0) \quad \  \text{in}\quad H^{1}(\R^N,\mathbb{C}^3),
\end{equation*}
where $\tilde{\kappa}=\left(\frac{a^2_1}{\|w_p\|^2_{L^2(\R^N)}} \right)^{\frac{p-2}{2-p\gamma_p}}$ and $w_p$ is the unique, positive solution of $-\Delta w+w=|w|^{p-2}w$.
\end{theo}

\begin{Rem}\rm
Theorem \ref{th1.2} together with Theorem \ref{th1.1} yields the multiplicity of standing waves for problem \eqref{eqA0.1}. This indicates that the coupling term not only makes the ground states stable, but also enriches the solutions set. See the first paragraph of Subsection \ref{sub:phys} for a description of what happens from a physical point of view.
\vskip1mm
\end{Rem}
\begin{Rem}\rm
The condition $\max\{a_1,a_2\}<D$ in Theorem \ref{th1.1} and Theorem \ref{th1.2} not only ensures that the corresponding energy functional $E$ admits a convex-concave geometry, but also guarantees the existence of a natural constraint (the Pohozaev-Nehari manifold, see later on in the paper), on which the critical points of $E$ are indeed nontrivial solutions to the problem \eqref{eqA0.2}. The condition $\alpha>\alpha_0$ is used for a better control of the energy level which excludes semi-trivial solutions. Point \textup{(ii)} of Theorem \ref{th1.2} draws an accurate semi-trivial limiting behavior of the excited states as portion of the mass vanishes. The transition from mass-supercritical to mass-critical regime dramatically changes the geometry of $E|_{S(a_1,a_2)}$, preventing the appearance of the excited state in the latter case. Moreover, if $p=2_*$, similarly to the proof of point \textup{(ii)} in Theorem \ref{th1.2},  the same semi-trivial limiting behavior of ground states obtained in Theorem \ref{th1.1} holds if and only if $a_1^2=\|w_p\|^2_{L^2(\R^N)}$. It is worth mentioning that similar semi-trivial limits of ground states for mass-critical Schr\"{o}dinger systems were obtained in  \cite{BS,GL}.
\end{Rem}

Based on the existence results on ground states and excited states, we can provide sufficient conditions for the global dynamics of solutions.\\
\noindent Firstly, with a control on the energy by means of the  excited state obtained in Theorem \ref{th1.2}, we show a global existence result. Let us define the Pohozaev functional $P$ by
\begin{equation}\label{poho-intro}
P(\vec{u}):=\sum^{3}_{i=1}\|\nabla u_i\|^2_2-\gamma_{p} \sum^{3}_{i=1}\|u_i\|^p_p-\frac{N\alpha}{2} \mathrm{Re} \int_{\R^N} u_1u_2\overline{u}_3dx.
\end{equation}
We have the following.
\begin{theo}\label{th1.4} Under the assumptions of Theorem \ref{th1.2}, let $\vec{\psi}$ be the solution of \eqref{eqA0.1} with initial datum $\vec{\psi}_0 \in S(a_1,a_2)$ such that $P(\vec{\psi}_0)>0$ and $E(\vec{\psi}_0)<E(\vec{v})$.
Then, $\vec{\psi}$ exists globally in time.
\end{theo}

Secondly, we are able to prove that under certain  conditions on the initial datum, finite time blowing-up solutions exist.

\begin{theo}\label{thm:blowup} Under the assumption of Theorem \ref{th1.2}, let $\vec{\psi}$ be the solution of \eqref{eqA0.1} with initial datum $\vec{\psi}_0 \in S(a_1,a_2)$, $P(\vec{\psi}_0)<0$ and $E(\vec{\psi}_0)<E(\vec v)$.
If $|x|\vec{\psi}_0\in L^2(\R^N,\mathbb{C}^3)$, the solution blows-up in finite time. The same conclusion holds true for $N=2,3$ for infinite variance solutions which are radial  provided $p\in(4,6)$ for $N=2$.
\end{theo}
The previous Theorem implies the following instability result.
\begin{corollary}\label{cor:insta}
The  standing wave    $\vec{\psi}(t,x)=\left(e^{i\lambda_1 t}v_1,e^{i\lambda_1 t}v_2,e^{i(\lambda_1+\lambda_2)t}v_3\right)$  constructed with $\vec v$ as in Theorem \ref{th1.2} is strongly unstable.
\end{corollary}

\begin{Rem}\rm
The set
\begin{equation*}
\Lambda_0:=\left\{\vec{u}\in S(a_1,a_2)  \hbox{ s.t. } P(\vec{u})>0 \ \ \text{and} \ \ E(\vec{u})< E(\vec{v})\right\}
\end{equation*}
is not empty and contains not only small initial data in the sense of the ${L^2(\R^N)}$-norm. Given $\gamma,\mu,\nu>0$, in the same manner we can look for solutions $(u_1,u_2,u_3)\in H^1(\R^N,\mathbb{C}^3)$ of \eqref{eqA0.2} satisfying the conditions $\|u_1\|^2_2=\gamma$, $\|u_2\|^2_2=\mu$, and $\|u_3\|^2_2=\nu$. Such solutions are of interest in physics and sometimes referred to as normalized solutions. In the present paper, we care more about solutions of \eqref{eqA0.2} with prescribed partial sum of  masses. This is not only because $Q_1(\vec{u})$ and $Q_2(\vec{u})$ are invariant with respect to the flow generated by \eqref{eqA0.1} but also because it is suitable for studying dynamics of \eqref{eqA0.1}.
\end{Rem}
\begin{Rem}\rm
The last remark is on the fact that similar results as the ones described above can be stated for  $\alpha<0$, provided one replaces $u_3$ by $-u_3$ in \eqref{eqA0.2}.
\end{Rem}

\subsection{Physical background and motivations}\label{sub:phys} The study of the model as described by equations in \eqref{eqA0.1} has a physical motivation, as the system \eqref{eqA0.1} is a simplified model of a quasilinear
Zakharov system related to the Raman amplification in a plasma. See  \cite{CC} for details. Roughly speaking, the Raman amplification is an instability phenomenon taking place when an incident laser field propagates into a plasma (see \cite{HAG} and the introduction in \cite{PA}). As explained in \cite{PA},  the laser  field, entering a plasma,  is backscattered by a Raman type process and the interaction of the two waves generates an electronic plasma wave. Then the three waves together produce a change in the ions' density which in turn affects the waves. This picture is described by three Schr\"odinger equations coupled  with a wave equation (i.e., a Zakharov type system) as follows:
\begin{equation}\label{eq:phys1}
\begin{cases}
\left(i\left(\partial_{t}+v_{C} \partial_{y}\right)+\alpha_{1} \partial_{y}^{2}+\alpha_{2} \Delta_{\perp}\right) A_{C}=\frac{b^{2}}{2} n A_{C}-\gamma(\nabla \cdot E) A_{R} e^{-i \theta}, \\
\left(i\left(\partial_{t}+v_{R} \partial_{y}\right)+\beta_{1} \partial_{y}^{2}+\beta_{2} \Delta_{\perp}\right) A_{R}=\frac{b c}{2} n A_{R}-\gamma\left(\nabla \cdot \overline{E}\right) A_{C} e^{i \theta}, \\
\left(i \partial_{t}+\delta_{1} \Delta\right) E=\frac{b}{2} n E+\gamma \nabla\left(\overline{A}_{R} A_{C} e^{i \theta}\right), \\
\left(\partial_{t}^{2}-v_{s}^{2} \Delta\right) n=a \Delta\left(|E|^{2}+b\left|A_{C}\right|^{2}+c\left|A_{R}\right|^{2}\right),
\end{cases}
\end{equation}
where $\theta=k_{1} y-k_1^2\delta_1 t$, $t \in \mathbb{R}$, $y \in \mathbb{R}$, and $\Delta_{\perp}=\partial_{x}^{2}+\partial_{z}^{2}$.
In this system, $A_C$ denotes the envelope of the incident laser field, $A_R$ is the backscattered Raman field, $E$ is the electronic-plasma wave and $n$ is the variation of ions' density.  We refer to \cite{CC, CC2} for a precise description of the physical coefficients appearing in the equations above. \medskip

After proving the local well-posedness of \eqref{eq:phys1}, in order to study the solitary waves towards an analysis of the global dynamics, the authors of \cite{CCO} needed to introduce some modifications  on \eqref{eq:phys1}, eventually leading to the system \eqref{eqA0.1} studied in this paper. For the reader's convenience and sake of clarity, we report here the few steps as in \cite{CCO} to derive the desired three NLS system.  \medskip

In \eqref{eq:phys1}, by writing $E=Fe^{i\theta}$, by considering a trivial density of ions, i.e., $n=0$, and by neglecting the $\nabla$ terms, the longitudinal dispersion terms $\partial^2_{y}$, and the transverse ones   $\Delta_{\perp}$,
one reduces to the simplified system
\begin{equation*}\label{eq1.1}
\begin{cases}
\left(i\partial_t+\alpha_{2} \Delta_{\perp}\right) A_{C}=-\gamma i k_{1} F A_{R}, \\
\left(i\partial_t+\beta_{2} \Delta_{\perp}\right) A_{R}=\gamma i k_{1} \overline{F} A_{C}, \\
\left(i\partial_t+\delta_{1} \Delta\right) F=i k_{1} \gamma \overline{A}_{R} A_{C}.
\end{cases}
\end{equation*}
In order to model nonlinear effects,  the other nonlinear terms as appearing in \eqref{eqA0.1} were added in \cite{CCO}, hence by a simple change of variables, and the introduction of the power-type nonlinear terms, one gets \eqref{eqA0.1}.

\subsection{Notations} In the paper, we use the following notations. $x\in\R^N$, $N\leq 3$, $t\in\R$, $L^p=L^p(\R^N)$ with norm $\|f\|_{L^p(\R^N)}=\|f\|_p$, $H^1(\R^N)$ is the usual Sobolev space, with $H^1(\R^N, \mathbb C^3)$ or $H^1(\R^N, \mathbb R^3)$ for vector valued functions, or $H^1(\R^N, \mathbb R)$ and $H^1(\R^N, \mathbb C)$ for scalar functions. $H^{-1}(\R^N)$ denote the dual space of $H^1(\R^N)$.  $\int_{\R^N} f dx $} is denoted simply by $\int f$. $\mathrm{Re}$ and $\mathrm{Im}$ stand for the real and imaginary part of a complex number, respectively, and $\overline z$ stands for the complex conjugate of $z$.

\section{Preliminaries}
In this section, we give some preliminaries useful for the rest of the paper.

\begin{Lem} \label{lem2.1}
Let $N\leq 3$,  $2_*\le p<  2^*$, and $(u_1,u_2,u_3)\in H^1(\R^N,\mathbb{C}^3)$ be a solution to \eqref{eqA0.2}. Then the following Pohozaev-Nehari identity holds true:
\begin{equation*} \label{equ2.1}
\sum^{3}_{i=1}\int |\nabla u_i|^2=\gamma_{p} \sum^{3}_{i=1}\int |u_i|^p+\frac{N\alpha}{2} \mathrm{Re} \int u_1u_2\overline{u}_3.
\end{equation*}
\end{Lem}
\begin{proof}
The proof is standard and we refer for example to the classical reference \cite{BL}.
\end{proof}

We now introduce the $L^2$-norm-preserving dilation operator
\begin{equation*}
s\star \vec{u}(x):=\left(s^{\frac{N}{2}}u_1(sx),s^{\frac{N}{2}}u_2(sx),s^{\frac{N}{2}}u_3(sx)\right)
\end{equation*}
with $s>0$. As $\lim\limits_{s\to \infty}E(s\star \vec{u})=-\infty$, we see that $\inf\limits_{\vec{u}\in S(a_1,a_2)}E(\vec{u})=-\infty$ for $2_*<p<2^*$. Furthermore, we introduce (see  \cite{BJS}) the Pohozaev-Nehari set
\begin{equation*}\label{def:poho}
\mathcal{P}_{a_1,a_2}:=\left\{\vec{u}\in S(a_1,a_2) \hbox{ s.t. }P(\vec{u})=\sum^{3}_{i=1}\|\nabla u_i\|^2_2-\gamma_{p} \sum^{3}_{i=1}\|u_i\|^p_p-\frac{N\alpha}{2} \mathrm{Re} \int u_1u_2\overline{u}_3=0\right\},
\end{equation*}
where $\gamma_p$ is given in \eqref{def:gammap}.\medskip

The  set $\mathcal{P}_{a_1,a_2}$ is related to the fiber maps
\begin{equation}\label{eq:fiber}
\Psi_{\vec{u}}(s)=E(s\star\vec{u})=\frac{s^2}{2}\sum^{3}_{i=1} \|\nabla u_i\|^2_2-\frac{s^{p\gamma_p}}{p}\sum^{3}_{i=1}\| u_i\|^p_p-s^{\frac{N}{2}}\alpha\mathrm{Re}\int u_1u_2\overline{u}_3.
\end{equation}
Indeed, we have $s\Psi'_{\vec{u}}(s)=P(s\star\vec{u})$.
Note that $\mathcal{P}_{a_1,a_2}$ can be divided into the disjoint union $\mathcal{ P}_{a_1,a_2}=\mathcal{ P}_{a_1,a_2}^+\cup \mathcal{ P}_{a_1,a_2}^0\cup \mathcal{ P}_{a_1,a_2}^-$, where
\begin{equation}\label{eq:c41}
\begin{aligned}
	\mathcal{ P}_{a_1,a_2}^+&:=\left\{\vec{u}\in \mathcal{ P}_{a_1,a_2}  \hbox{ s.t. } \Psi_{\vec{u}}''(1)>0\right\},\\
	\mathcal{ P}_{a_1,a_2}^0&:=\left\{\vec{u}\in \mathcal{ P}_{a_1,a_2}  \hbox{ s.t. } \Psi_{\vec{u}}''(1)=0\right\},\\
	\mathcal{ P}_{a_1,a_2}^-&:=\left\{\vec{u}\in \mathcal{ P}_{a_1,a_2}  \hbox{ s.t. } \Psi_{\vec{u}}''(1)<0\right\}.
\end{aligned}
\end{equation}
We first study the case $2_*<p<2^*$, namely the mass-energy intercritical case.
To show that the energy functional $E|_{S(a_1,a_2)}$ has a concave-convex geometry (i.e., a structure with a local minimum and a global maximum, where the local minimum is strictly less than zero and the global maximum is strictly greater than zero - see Lemma \ref{lem2.11} below), we introduce the following constraint:
\begin{equation}\label{eq1.2}
\mathcal{M}:=\left\{(u_1,u_2,u_3)\in H^1(\R^N,\mathbb{C}^3) \quad \hbox{ s.t. } \quad \mathrm{Re}\int u_1u_2\overline{u}_3 >0\right\}.
\end{equation}
In the spirit of Soave \cite{Soave1} and Wei and Wu \cite{WW}, for $\vec{u}\in \mathcal{M}$, we see that the presence of the mass-subcritical terms $\displaystyle\mathrm{Re}\int u_1u_2\overline{u}_3$ induces a convex-concave geometry of $E|_{S(a_1,a_2)}$ if $\alpha>0$ and $a_1,a_2>0$ are small.
For $\vec{u}\in S(a_1,a_2)$, we have $\|u_1\|_2\le a_1$, $\|u_2\|_2\le a_2$ and $\|u_3\|_2\le \min\{a_1,a_2\}$.
By Gagliardo-Nirenberg inequality and Young inequality, we have
\begin{equation}\label{z5}
\begin{aligned}
\frac{1}{p}\sum^{3}_{i=1}\|u_i\|^p_p&\le \frac{C^p(N,p)}{p}\left(\sum^{2}_{i=1}a^{p(1-\gamma_p)}_i\|\nabla u_i\|^{p\gamma_p}_2+\max\left\{a^{p(1-\gamma_p)}_1,a^{p(1-\gamma_p)}_2\right\}\|\nabla u_3\|^{p\gamma_p}_2\right)\\
&\le A_1\left(\sum_{i=1}^3\|\nabla u_i\|^2_2\right)^{\frac{p\gamma_p}{2}},
\end{aligned}
\end{equation}
where $A_1:=\frac{C^p(N,p)}{p}\left(\max\{a_1,a_2\}\right)^{p(1-\gamma_p)}$.
Similarly, we have
\begin{equation}\label{z4}
\left|\alpha\mathrm{Re}\int u_1u_2\overline{u}_3\right|\le 
A_2\left(\sum_{i=1}^3\|\nabla u_i\|^2_2\right)^{\frac{N}{4}},
\end{equation}
where $A_2:=\frac{\alpha  C^3(N,p)}{3^{\frac{N}{4}}}\left(\max\{a_1,a_2\}\right)^{\frac{6-N}{2}}$.
Then, combining \eqref{z5} and \eqref{z4} with the definition of the energy,  we get
\begin{equation}\label{b3}
\begin{aligned}
E(\vec{u})&\ge \frac{1}{2}\sum^{3}_{i=1}\|\nabla u_{i}\|^2_2-A_1\left(\sum^{3}_{i=1}\|\nabla u_{i}\|^2_2\right)^{\frac{p\gamma_p}{2}}-A_2\left(\sum^{3}_{i=1}\|\nabla u_{i}\|^2_2\right)^{\frac{N}{4}}\\
&=h\left(\left(\sum^{3}_{i=1}\|\nabla u_{i}\|^2_2\right)^{\frac{1}{2}}\right),
\end{aligned}
\end{equation}
where
\begin{equation}\label{def:func-h}
h(\rho)=\frac{\rho^2}{2}-A_1\rho^{p\gamma_p}-A_2\rho^{\frac{N}{2}}.
\end{equation}
\vskip2mm

The next Lemma  below  shows that the functional $E$  has a concave-convex structure on $S(a_1,a_2)$.

\begin{Lem}\label{lem2.11}
Let $N\le 3$,  $2_*<p < 2^*$, and $\alpha,a_1, a_2>0$. Let $D$ be as in \eqref{def:D} and $h(\rho)$ as in \eqref{def:func-h}.\smallskip

\noindent \textup{(i)} If $\max\{a_1,a_2\}< D$, then $h(\rho)$ has a local minimum at negative level and a global maximum at positive level. Moreover, there exist $R_0=R_0(a_1,a_2)$, $R_1=R_1(a_1,a_2)$, and $\rho^*$ such that, $R_0<\max\{a_1,a_2\}D^{-1}\rho^*<\rho^*<R_1$, and
\begin{equation*}
h(R_0)=h(R_1)=0, \quad h(\rho)>0 \iff \rho \in (R_0,R_1).
\end{equation*}\smallskip
\noindent \textup{(ii)} If $\max\{a_1,a_2\}=D$,  then $h(\rho)$ has a local minimum at negative level and a global maximum at level zero.
Moreover, we have
\[
h(\rho^*)=0 \quad \hbox{ and } \quad h(\rho)<0 \iff \rho\in(0, \rho^*) \cup (\rho^*,+\infty).
\]
\end{Lem}
\begin{proof}
\textup{(i)} We first prove that $h(\rho)$ has exactly two critical points. Indeed,
\[
h'(\rho)=0 \Longleftrightarrow \hat h(\rho)=\frac{NA_2}{2}, \quad \mbox{with} \quad \hat h(\rho)=\rho^{2-\frac{N}{2}}-p\gamma_pA_1\rho^{p\gamma_p-\frac{N}{2}}.
\]
By defining $\bar{\rho}=\left(\frac{4-N}{p\gamma_p(2p\gamma_p-N)A_1}\right)^{\frac{1}{p\gamma_p-2}}$, we have that $\hat h(\rho)$ is increasing on $[0,\bar{\rho})$ and decreasing on $(\bar{\rho},+\infty)$.
Since $2<p\gamma_p$, we get
\[
\max_{\rho\geq0}\hat h(\rho)=\hat h(\bar{\rho})
=\frac{2p\gamma_p-4}{2p\gamma_p-N}\left(\frac{4-N}{p\gamma_p(2p\gamma_p-N)A_1}\right)^{\frac{4-N}{2p\gamma_p-4}}
>\frac{NA_2}{2}
\]
if and only if
\begin{equation*}
\begin{aligned}
\max\{a_1,a_2\}<D_0:=
\left(\frac{3^{\frac{N}{4}}}{\alpha C^3(N,p)}\frac{2(2p\gamma_p-4)}{N(2p\gamma_p-N)}\right)^{\frac{N(p-2)-4}{4(p-3)}}
\left(\frac{4-N}{\gamma_p(2p\gamma_p-N)C^p(N,p)}\right)^{\frac{4-N}{4(p-3)}}.
\end{aligned}
\end{equation*}
As $\lim\limits_{\rho\to 0^+}\hat h(\rho)=0^+$ and $\lim\limits_{\rho\to +\infty}\hat h(\rho)=-\infty$, we see  that $h(\rho)$ has exactly two critical points if $\max\{a_1,a_2\}<D_0$.

\noindent Note that $
h(\rho)>0 \Longleftrightarrow \tilde h(\rho)>A_2$ with $\tilde h(\rho)=\frac{1}{2}\rho^{2-\frac{N}{2}}-A_1\rho^{p\gamma_p-\frac{N}{2}}.$
It is not difficult to check that $ \tilde h(\rho)$ is increasing on $[0,\rho_0)$ and decreasing on $(\rho_0,+\infty)$, where $\rho_0=\left(\frac{4-N}{2(2p\gamma_p-N)A_1}\right)^{\frac{1}{p\gamma_p-2}}$.
We have
\[
\max_{\rho\geq0} \tilde h(\rho)= \tilde h(\rho_0)
=\frac{p\gamma_p-2}{2p\gamma_p-N}\left(\frac{4-N}{2(2p\gamma_p-N)A_1}\right)^{\frac{4-N}{2p\gamma_p-4}}
>A_2
\] provided that
\begin{equation*}
\begin{aligned}
\max\{a_1,a_2\}<D:=
\left(\frac{3^{\frac{N}{4}}}{\alpha C^3(N,p)}\frac{p\gamma_p-2}{2p\gamma_p-N}\right)^{\frac{N(p-2)-4}{4(p-3)}}
\left(\frac{p(4-N)}{2(2p\gamma_p-N)C^p(N,p)}\right)^{\frac{4-N}{4(p-3)}}.
\end{aligned}
\end{equation*}
We have $h(\rho)>0$ on an open interval $(R_0,R_1)$ if and only if $\max\{a_1,a_2\}<D$. We claim that $D<D_0$. To this purpose, we only need to prove that $\left(\frac{4}{N}\right)^{N(p-3)}\left(\frac{1}{p-2}\right)^{4-N}>1$ holds. As in \cite[Lemma 5.2]{Soave1}, by letting $z=\frac{4}{N}$ and $y=p-2$,  we have
\begin{equation*}
\left(\frac{4}{N}\right)^{N(p-3)}\left(\frac{1}{p-2}\right)^{4-N}>1\Longleftrightarrow z^{y-1}>y^{z-1}.
\end{equation*} Since $\frac{\log z}{z-1}$ is a monotone decreasing function for $z>0$, we have $D<D_0$.

\noindent If $\max\{a_1,a_2\}<D$, combining $\lim_{\rho\to 0^+}h(\rho)=0^{-}$ and $\lim_{\rho\to +\infty}h(\rho)=-\infty$, we see that $h(\rho)$ has a local minimum point at negative level in $(0,R_0)$ and a global maximum point at positive level in $(R_0,R_1)$. Define
\begin{equation}\label{rhostar}
\rho^*:=\left(\frac{p(4-N)}{2(2p\gamma_p-N)C^p(N,p)}\right)^{\frac{1}{p\gamma_p-2}}D^{-\frac{p(1-\gamma_p)}{p\gamma_p-2}},
\end{equation}
then  $\rho^*<\rho_0$. By direct calculations, we have
\begin{equation*}
\tilde h(\rho^*)>\frac{1}{2}(\rho^*)^{2-\frac{N}{2}}-\frac{C(N,p)D^{p(1-\gamma_p)}}{p}(\rho^*)^{p\gamma_p-\frac{N}{2}}=\frac{\alpha C^3(N,p)D^{\frac{6-N}{2}}}{3} >A_2,
\end{equation*}
then $h(\rho^*)>0$ and $\rho^*>R_0$. Note that $\rho^*$ is independent of $a_1,a_2$. In addition, it holds that
\begin{equation*}
\tilde h\left(\frac{\max\{a_1,a_2\}}{D}\rho^*\right)>\left(\frac{\max\{a_1,a_2\}}{D}\right)^{\frac{4-N}{2}}\frac{\alpha C^3(N,p) D^{\frac{6-N}{2}}}{3} >A_2.
\end{equation*}
\vskip2mm
\noindent \textup{(ii)}  As in the proof of \textup{(i)}, we have $
R_0=\bar{\rho}=\rho_{0}=\rho^*=R_1,\quad \tilde h(\rho_{0})=A_2,\quad \hat h(\bar{\rho})>\frac{N}{2}A_2.$
\end{proof}
Next, we study the structure of the manifold
\begin{equation*}\label{def:bar-p}
\bar{\mathcal{P}}_{a_1,a_2}:=\mathcal{P}_{a_1,a_2}\cap \mathcal{M}.
\end{equation*}
We will observe that a critical point for the energy functional $E$ on $\bar{\mathcal{P}}_{a_1,a_2}$ is a critical point for the functional $E$ on $S(a_1,a_2)$. Hence, $\bar{\mathcal{P}}_{a_1,a_2}$ is a natural constraint.
\begin{Lem}\label{lem2.2}
Let $N\le 3$,  $2_*<p < 2^*$, and $\alpha,a_1, a_2>0$. If $\max\{a_1,a_2\}\le D$, then $\mathcal{P}^0_{a_1,a_2}=\emptyset$, and the set $\bar{\mathcal{P}}_{a_1,a_2}$ is a $ C^1$-submanifold of codimension 1 in $S(a_1,a_2)$.
\end{Lem}
\begin{proof}
We adopt an argument by  Soave from \cite{Soave1}. It is sufficient to prove that $\mathcal{P}^0_{a_1,a_2}$ is empty. Indeed, a consequence of $\mathcal{P}^0_{a_1,a_2}=\emptyset$ is that $\bar{\mathcal{P}}_{a_1,a_2}$ is a $C^1$-submanifold of codimension 1 in $S(a_1,a_2)$.
Assume by contradiction that there exists a $\vec{u} \in \mathcal{P}^0_{a_1,a_2}$ such that $P(\vec{u})=0$, thus
\begin{equation*}
\Psi''_{\vec{u}}(1)=\sum^{3}_{i=1}\int \left(2|\nabla u_i|^{2}-p\gamma^2_p|u_i|^p\right)
-\frac{N^2\alpha}{4}\mathrm{Re}\int u_1u_2\overline{u}_3=0.
\end{equation*}
Let
\begin{equation*}
\begin{aligned}
f(y):&=y\Psi'_{\vec{u}}(1)-\Psi''_{\vec{u}}(1)\\
&=(y-2)\sum^{3}_{i=i}\int |\nabla u_i|^2-(y-p\gamma_p)\gamma_p\sum^{3}_{i=i}\int |u_i|^{p}-\left(y-\frac{N}{2}\right)\frac{N}{2}\alpha\mathrm{Re}\int u_1u_2\overline{u}_3,\\
\end{aligned}
\end{equation*}
and observe that $f(y)=0, ~\forall y\in \R$. Therefore, it follows from $f\left(\frac{N}{2}\right)=0$ that
\begin{equation}\label{z6}
\left(2-\frac{N}{2}\right)\sum^{3}_{i=2}\|\nabla u_i\|^2_2=\gamma_p\left(p\gamma_p-\frac{N}{2}\right)\sum^{3}_{i=2}\|u_i\|^p_p.
\end{equation}
By \eqref{z5} and $\eqref{z6}$, we have
\begin{equation*}
\left(\sum^{3}_{i=1}\|\nabla u_i\|^2_2\right)^{\frac{1}{2}}\ge \left(\frac{4-N}{\gamma_p(2p\gamma_p-N)C^p(N,p)}\right)^{\frac{1}{p\gamma_p-2}}
\left(\max\{a_1,a_2\}\right)^{-\frac{^{p(1-\gamma_p)}}{p\gamma_p-2}}.
\end{equation*}
Since $f(p\gamma_p)=0$, we get
\begin{equation*}
\begin{aligned}
(p\gamma_p-2)&=\left(p\gamma_p-\frac{N}{2}\right)\frac{N}{2}\left(\sum^3\limits_{i=1}\|\nabla u_i\|^2_2\right)^{-1}\alpha\mathrm{Re}\int u_1u_2\overline{u}_3\\
&\le \left(p\gamma_p-\frac{N}{2}\right)\frac{NA_2}{2}\left(\frac{4-N}{\gamma_p(2p\gamma_p-N)C^p(N,p)}\right)^{\frac{N-4}{2p\gamma_p-4}}
\left(\max\{a_1,a_2\}\right)^{\frac{p(1-\gamma_p)(4-N)}{2p\gamma_p-4}},
\end{aligned}
\end{equation*}
which is a contradiction with respect  to the hypothesis $\max\{a_1,a_2\}\le D<D_0$.\vskip2mm

\noindent We omit the proof that $\bar{\mathcal{P}}_{a_1,a_2}$ is a smooth manifold of codimension 1 on $S(a_1,a_2)$.
\end{proof}

\begin{Lem}\label{lem2.3}Let $N\le 3$,  $2_*<p < 2^*$, and $\alpha,a_1, a_2>0$. If $\max\{a_1,a_2\}< D$, for $\vec{u}\in S(a_1,a_2)\cap \mathcal{M}$, then the function $\Psi_{\vec{u}}(s)$ has exactly two critical points $s_{\vec{u}}<\sigma_{\vec{u}}\in\R$ and two zeros $c_{\vec{u}}<d_{\vec{u}}$ with $s_{\vec{u}}<c_{\vec{u}}<\sigma_{\vec{u}}<d_{\vec{u}}$. Moreover, we have the properties below:\smallskip

\noindent \textup{(i)}  $s_{\vec{u}}\star\vec{u}\in\mathcal{P}^{+}_{a_1,a_2}$ and $\sigma_{\vec{u}}\star\vec{u}\in\mathcal{P}^{-}_{a_1,a_2}$. Moreover,  if $s\star\vec{u}\in \mathcal{P}_{a_1,a_2}$, then either $s=s_{\vec{u}}$   or $s=\sigma_{\vec{u}}$;\smallskip

\noindent\textup{(ii)}  $\displaystyle s_{\vec{u}}<R_0\left(\sum^{3}_{i=1}\|\nabla u_i\|^2_2\right)^{-\frac{1}{2}}$ and
\begin{equation*}
\Psi_{\vec{u}}(s_{\vec{u}})=\inf\left\{\Psi_{\vec{u}}(s):s\in\left(0,R_0\left(\sum^{3}_{i=1}\|\nabla u_i\|^2_2\right)^{-\frac{1}{2}}\right)\right\}<0;
\end{equation*}
\smallskip
\noindent\textup{(iii)}  $E\left(\sigma_{\vec{u}}\star\vec{u}\right)=\max\limits_{s\in\R^+}E\left(s\star\vec{u}\right)>0$;\smallskip

\noindent\textup{(iv)} the maps $\vec{u} \mapsto s_{\vec{u}} \in \mathbb{R}^+$
and $\vec{u} \mapsto \sigma_{\vec{u}} \in \mathbb{R}^+$ are of class $ C^1$.
\end{Lem}
\begin{proof}
Let $\vec{u}\in S(a_1,a_2)$, we have $s\star\vec{u}\in \mathcal{P}_{a_1,a_2}$ if and only if $\Psi'_{\vec{u}}(s)=0$, $\Psi_{\vec u}$ defined in \eqref{eq:fiber}.
By \eqref{b3}-\eqref{def:func-h}, we get
\begin{equation*}
\Psi_{\vec{u}}(s)=E\left(s\star\vec{u}\right)\ge h\left(s\left(\sum^{3}_{i=1}\|\nabla u_i\|^2_2\right)^{\frac{1}{2}}\right).
\end{equation*}
If $\max\{a_1,a_2\}<D$, from point \textup{(i)} of Lemma \ref{lem2.1}, $\Psi_{\vec{u}}(s)$ is positive in  the interval \[
\left(R_0\left(\sum^{3}_{i=1}\|\nabla u_i\|^2_2\right)^{-\frac{1}{2}},R_1\left(\sum^{3}_{i=1}\|\nabla u_i\|^2_2\right)^{-\frac{1}{2}} \right),\] and we have the asymptotic behavior $\lim\limits_{s\to 0^+}\Psi_{\vec{u}}(s)=0^-$, $\lim\limits_{s\to +\infty}\Psi_{\vec{u}}(s)=-\infty$, thus we can see that $\Psi_{\vec{u}}(s)$ has a local minimum point $s_{\vec{u}}$ in the interval $\left(0,R_0\left(\sum^{3}_{i=1}\|\nabla u_i\|^2_2\right)^{-\frac{1}{2}} \right)$ and a global maximum point $\sigma_{\vec{u}}$ in the interval $\left(R_0\left(\sum^{3}_{i=1}\|\nabla u_i\|^2_2\right)^{-\frac{1}{2}},R_1\left(\sum^{3}_{i=1}\|\nabla u_i\|^2_2\right)^{-\frac{1}{2}} \right)$. It follows from Lemma \ref{lem2.1} that $\Psi_{\vec{u}}(s)$ has no other critical points.
\vskip2mm
\noindent Since $\Psi''_{\vec{u}}(s_{\vec{u}})\ge 0$, $\Psi''_{\vec{u}}(\sigma_{\vec{u}})\le 0$ and $\mathcal{P}^0_{a_1,a_2}=\emptyset$, we know that $s_{\vec{u}}\star\vec{u}\in \mathcal{P}^+_{a_1,a_2}$ and $\sigma_{\vec{u}}\star\vec{u}\in \mathcal{P}^-_{a_1,a_2}$. By the monotonicity and the behavior at infinity of $\Psi_{\vec{u}}(s)$, we get that $\Psi_{\vec{u}}(s)$ has exactly two zeros $c_{\vec{u}}$ and $d_{\vec{u}}$ with $s_{\vec{u}}<c_{\vec{u}}<\sigma_{\vec{u}}<d_{\vec{u}}$. Thus, the conclusions \textup{(i)},\textup{(ii)}, and \textup{(iii)} follow from the facts above. Point \textup{(iv)} is a consequence of the  Implicit Function Theorem on the $ C^1$ function $g: \R \times S(a_1,a_2) \mapsto \R^+$ defined by $g=g_{\vec u}(s)= \Psi_{\vec{u}}'(s)$ as $g_{\vec{u}}(s_{\vec{u}}) =0$ and $\partial_s g_{\vec{u}}(s_{\vec{u}}) = \Psi_{\vec{u}}''(s_{\vec{u}})>0$, and similarly for $\vec{u} \mapsto \sigma_{\vec{u}} \in \mathbb{R}^+$. 
\end{proof}

\section{Proof of Theorem \ref{th1.1}}

In this section, we give a proof of Theorem \ref{th1.1}, and we divide it into two cases: $p=2_*$ and $2_*<p<2^*$. We first prove several results eventually leading to the conclusions of Theorem \ref{th1.1}.
\subsection{Mass-energy intercritical case}
Let $c>0$, and for $N\leq 3$ we consider $2_*<p<2^*$. We introduce the following complex valued equation:
\begin{equation} \label{b1}
\begin{cases}
-\Delta u+\lambda u=|u|^{p-2}u, \quad u\in H^1(\R^N, \mathbb C),\\
\displaystyle \int |u|^2=c^2.
\end{cases}
\end{equation}
From \cite{COZ,KKM,CZZ,Soave1}}, the solutions of \eqref{b1} corresponds to the critical points of the functional $J: H^1(\R^N,\mathbb{C})\to \R$,
\begin{equation}\label{z2}
 J(u)=\frac{1}{2}\int |\nabla u|^2-\frac{1}{p}\int |u|^p,
\end{equation}
constrained on the sphere
\[
S(c)=\{u\in H^1(\R^N, \mathbb C) \quad  \hbox{s.t.} \quad  \|u\|^2_2=c^2\},
\] and the parameter $\lambda$ appears as a Lagrange multiplier. We introduce the Pohozaev-Nehari constraint for the single equations \eqref{b1}
 \begin{equation*}\label{def:p-c}
\mathcal{P}_c:=\left\{u\in H^1(\R^N,\mathbb{C})\cap S(c)  \quad \hbox{s.t.} \quad \|\nabla u\|^2_2=\gamma_p\|u\|^p_p \right\},
\end{equation*}
recalling that $\gamma_p=\frac{N(p-2)}{2p}$. Define
\begin{equation}\label{c1}
m(c)=\inf_{u\in \mathcal{P}_c}J(u)>0.
\end{equation}
The next Lemma, see  \cite[Lemma 2.3]{HS},  ensures that the infimum $m(c)$ above is the same if we restrict to real functions.

\begin{Lem}\label{lem2.12}
Let $c>0$, $N\le 3$, and   $2_*<p < 2^*$. We have that
\begin{equation*}
m(c)=\inf_{u\in H^1(\R^N,\R)\cap \mathcal{P}_c}J(u),
\end{equation*}
and $m(c)$ is strictly decreasing with respect to $c$. Moreover, any normalized solution of \eqref{b1} has the form $e^{i\sigma}U$, where $\sigma\in\R$ and $U$ is a positive, radial decreasing normalized solution of \eqref{b1}.
\end{Lem}

Let us introduce the set
\begin{equation*}
B_{\rho^*}:=\left\{\vec{u}\in H^1(\R^N,\mathbb{C}^3) \quad \hbox{s.t.} \quad \left(\sum^{3}_{i=1}\|\nabla u_{i}\|^2_2\right)^{\frac{1}{2}}<\rho^*\right \}
\end{equation*}
and
\begin{equation*}
V(a_1,a_2):=S(a_1,a_2)\cap B_{\rho^*}\cap \mathcal{M},
\end{equation*}
where $\mathcal{M}$ is defined in \eqref{eq1.2} and $\rho^*$ in \eqref{rhostar}.
Thus, we can define the following minimization problem:  for any positive $a_1$ and $a_2$ such that $\max\{a_1,a_2\}< D$, let
\begin{equation*}\label{def:m12}
m(a_1,a_2) := \inf_{\vec{u} \in V(a_1,a_2)} E(\vec{u}).
\end{equation*}

\begin{Lem}\label{lem2.4}
Let $N\le 3$,  $2_*<p < 2^*$, and $\alpha,a_1, a_2>0$. If $\max\{a_1,a_2\}< D$, the set $\mathcal{P}^+_{a_1,a_2}$ is contained in $V(a_1,a_2)$ and
\begin{equation}\label{m1}
m(a_1,a_2)=m^+(a_1,a_2):=\inf_{\vec{u}\in\mathcal{P}^+_{a_1,a_2}\cap \mathcal{M}}E(\vec{u})=\inf_{\vec{u}\in\bar{\mathcal{P}}_{a_1,a_2}}E(\vec{u})<0.
\end{equation}
Moreover, there exists $\varepsilon_0>0$ such that for any $0<\varepsilon<\varepsilon_0$
\begin{equation*}
m(a_1,a_2)<\inf_{\vec u\in S(a_1,a_2)\cap (B_{\rho^*}\setminus B_{\rho^*-\varepsilon})} E(\vec u).
\end{equation*}
\end{Lem}
\begin{proof}
For $\vec{u}\in V(a_1,a_2)$, we have
$E(\vec{u})\ge h\left(\left(\sum^{3}_{i=1}\|\nabla u_i\|^2_2\right)^{\frac{1}{2}}\right)\ge \min_{\rho\in [0,R_0]}h(\rho)>-\infty,$ where $R_0$ and $h$ are given in Lemma \ref{lem2.11}.
For a function $\vec{u}\in S(a_1,a_2)\cap \mathcal{M}$, there exists $s_0>0$ small enough such that $s_0\star \vec{u}\in B_{\rho_0}$ and $E(s_0\star \vec{u})<0$. Hence, we get $m(a_1,a_2)\in (-\infty, 0)$.
From Lemma \ref{lem2.3}, we have $\mathcal{P}^+_{a_1,a_2}\cap \mathcal{M}\subset V(a_1,a_2)$, and then $m(a_1,a_2)\le \inf\limits_{\mathcal{P}^+_{a_1,a_2}\cap \mathcal{M}}E$.
In addition, if $\vec{u}\in V(a_1,a_2)$, $s_{\vec{u}}\star \vec{u}\in \mathcal{P}^+_{a_1,a_2}\subset V(a_1,a_2)$, we get
\begin{equation*}
E\left(s_{\vec{u}}\star\vec{u}\right)=\min\left\{E(s\star\vec{u}): s\in \R^+\ \ \text{and}\ \ s\star\vec{u}\in V(a_1,a_2)\right\}\le E(\vec{u}),
\end{equation*}
and it follows that $\inf\limits_{\mathcal{P}^+_{a_1,a_2}\cap \mathcal{M}}E\le m(a_1,a_2)$. By Lemma \ref{lem2.3}, $E(\vec{u})>0$ on $\mathcal{P}^-_{a_1,a_2}$, so we conclude that $m(a_1,a_2)=\inf\limits_{\vec{u}\in\bar{\mathcal{P}}_{a_1,a_2}}E(\vec{u})=\inf\limits_{\vec{u}\in\mathcal{P}^+_{a_1,a_2}\cap \mathcal{M}}E(\vec{u}).$
\vskip1mm

\noindent There exists $\varepsilon>0$ small enough such that, if $\rho \in [\rho^*-\varepsilon,\rho^*]$, we have $h(\rho)\ge \frac{m(a_1,a_2)}{2}$, and then
\begin{equation*}
E(\vec{u})\ge h\left(\sum^{3}_{i=1}\|\nabla u_i\|^2_2\right)\ge \frac{m(a_1,a_2)}{2}>m(a_1,a_2),
\end{equation*}
for any $\vec{u}\in S(a_1,a_2)$ and $\displaystyle \rho^*-\varepsilon\le \sum^{3}_{i=1}\|\nabla u_j\|^2_2\le \rho^*$, where in the last inequality we used the fact that $m$ is negative.
\end{proof}

Let $\vec{u}$ belong to $H^1(\R^N,\mathbb{C}^3)$.  $E(|\vec{u}|)\le E(\vec{u})$, and by the symmetric rearrangement, see \cite{BF,LE},
\[
\|\nabla |u_i|^*\|_2\le \|\nabla |u_i|\|_2\le \|\nabla u_i\|_2, \quad \||u_i|^*\|_p=\|u_i\|_p, \]
and
\[
\int |u_1||u_2||u_3|\le \int |u_1|^*|u_2|^*|u_3|^*,
\]
where $|u_i|^*$ is the Schwarz  symmetric rearrangement of $|u_i|$, for $i=1,2,3.$ Then $E(|\vec{u}|^*)\le E(|\vec{u}|)\le E(\vec{u})$, where the short  notation $|\vec{u}|^*$ stands for  $|\vec{u}|^*=(|u_1|^*,|u_2|^*,|u_3|^*)$. Let us consider  $(v_1,v_2,v_3)\in H^1(\R^N,\R^3)$ a solution to the  system \eqref{eqA0.2}, namely
\begin{equation*}\label{b4}
\begin{cases}
-\Delta v_{1}+ \lambda_1v_{1}=\left|v_{1}\right|^{p-2} v_{1}+\alpha v_{3} v_{2},\\
-\Delta v_{2}+ \lambda_2v_{2}=\left|v_{2}\right|^{p-2} v_{2}+\alpha v_{3} v_{1}, \\
-\Delta v_{3}+ (\lambda_1+\lambda_2)v_{3}=\left|v_{3}\right|^{p-2} v_{3}+\alpha v_{1} v_{2}.
\end{cases}
\end{equation*}
Denote
$\mathcal{P}_{r,a_1,a_2}:=\left\{\vec{v}\in H^1_r(\R^N,\R^3)\cap S(a_1,a_2) \quad \hbox{s.t.} \quad P(\vec{v})=0\right\}$,
and
\[
\mathcal{P}^{+}_{r,a_1,a_2}:=H^1_r(\R^N,\R^3)\cap \mathcal{P}^{+}_{a_1,a_2}.\]
The notation $H^1_r(\R^N,\R^3)$ denotes the subspace of functions in $H^1(\R^N,\R^3)$ which are radially symmetric.
We set
\begin{equation}\label{y2}
m^{+}_{r}(a_1,a_2):=\inf\limits_{\vec{u}\in \mathcal{P}^{+}_{r,a_1,a_2}\cap \mathcal{M}}E(\vec{u}),
\end{equation}
and $W^{+}_r:=\left\{\vec{u}\in H^1_r(\R^N,\R^3)\cap S(a_1,a_2) \quad \hbox{s.t.} \quad E(\vec{u})=m^{+}_r(a_1,a_2)   \right\}$. \smallskip

We have the following.
\begin{Lem}\label{lem2.5}
Let $N\le 3$,  $2_*<p < 2^*$, and $\alpha,a_1, a_2>0$. If $\max\{a_1,a_2\}<D$, then
\begin{equation*}
m^{+}_{r}(a_1,a_2)=\inf_{\vec{u}\in \mathcal{P}^{+}_{r,a_1,a_2}\cap \mathcal{M}} E(\vec{u})=\inf_{\vec{u}\in\mathcal{P}^{+}_{a_1,a_2}\cap \mathcal{M}} E(\vec{u}).
\end{equation*}
Moreover, if $\inf\limits_{\vec u \in \mathcal{P}^{+}_{a_1,a_2}\cap \mathcal{M}} E(\vec u)$ is reached, it is reached by a Schwartz radially symmetric function. More precisely, $\inf\limits_{\vec u\in \mathcal{P}^{+}_{a_1,a_2}\cap \mathcal{M}} E(\vec u)$ is reached by   $(e^{i\theta_1}w_1, e^{i\theta_1}w_2, e^{i(\theta_1+\theta_2)}w_3 )$ where $\vec{w}$ is the minimizer for $\inf\limits_{\vec u\in \mathcal{P}^{+}_{r,a_1,a_2}} E(\vec u)$ and $(\theta_1,\theta_2)\in \R^2$.
\end{Lem}
\begin{proof}
It follows from $\mathcal{P}^{+}_{r,a_1,a_2}\subset \mathcal{P}^{+}_{a_1,a_2}$ that $\displaystyle\inf\limits_{\vec u\in\mathcal{P}^{+}_{r,a_1,a_2}\cap \mathcal{M}}  E\ge\inf\limits_{\vec u\in \mathcal{P}^{+}_{a_1,a_2}\cap\mathcal{M}}  E(\vec u)$.
From Lemma \ref{lem2.3}, for any $\vec{u}\in S(a_1,a_2)\cap\mathcal{M}$, there exists $s^{+}_{\vec{u}}\in \R^+$ such that $s^{+}_{\vec{u}}\star\vec{u}\in\mathcal{P}^{+}_{a_1,a_2}$, and
\begin{equation*}
\inf_{\vec{u}\in\mathcal{P}^{+}_{a_1,a_2}\cap \mathcal{M}} E(\vec{u})=\inf_{\vec{u}\in S(a_1,a_2)\cap\mathcal{M}}\min _{0<\sigma\le s^+_{\vec{u}}} E(\sigma\star\vec{u}).
\end{equation*}
For $\vec{u}\in S(a_1,a_2)$, let $\vec{w}\in S_r(a_1,a_2)$ be the Schwarz rearrangement of $(|u_1|,|u_2|,|u_3|)$, i.e. $(w_1,w_2,w_3):=\left(|u_1|^*,|u_2|^*,|u_3|^*\right)$. Then,  for all $\sigma>0$, $E(\sigma\star \vec{w})\le E(\sigma\star \vec{u})$. Recalling that $\Psi'_{\vec{u}}(\sigma)=P(\sigma\star \vec{u})$, see \eqref{eq:fiber}, we have
\begin{equation*}
\lim_{\sigma\to0^+}\Psi'_{\vec{w}}(\sigma)\le \lim_{\sigma\to0^+}\Psi'_{\vec{u}}(\sigma)<0 \quad \text{and} \quad \Psi''_{\vec{w}}(\sigma)\le \Psi''_{\vec{u}}(\sigma) \quad \forall \ \sigma>0.
\end{equation*}

\noindent It follows that $-\infty<s^+_{\vec{u}}\le s^+_{\vec{w}}$.
Therefore, we have
\begin{equation*}
\min _{0<\sigma<s^+_{\vec{w}}} E(\sigma\star\vec{w})\le \min _{0<\sigma<s^+_{\vec{u}}} E(\sigma\star\vec{u}),
\end{equation*}
and then $\inf\limits_{\vec u\in\mathcal{P}^{+}_{r,a_1,a_2}\cap \mathcal{M}} E(\vec u)\le\inf\limits_{\vec u\in\mathcal{P}^{+}_{a_1,a_2}\cap\mathcal{M}} E(\vec u)$.
\vskip2mm
\noindent First, we set $\vec{v}:=(e^{i\theta_1}w_1,e^{i\theta_2}w_2,e^{i(\theta_1+\theta_2)}w_3)$, where $\theta_1,\theta_2\in \R$ and $E(\vec{w})=m^{+}_r(a_1,a_2)$. Then, $\vec{v}\in S(a_1,a_2)$ and
\begin{equation*}
E(\vec{v})=\frac{1}{2}\sum^{3}_{i=1}\|\nabla w_i\|^2_2-\frac{1}{p}\sum^{3}_{i=1}\|w_i\|^p_p-\alpha \mathrm{Re} \int e^{i\theta_1}w_1e^{i\theta_2}w_2e^{-i(\theta_1+\theta_2)}w_3=E(\vec{w}).
\end{equation*}
Thus, $\left\{(e^{i\theta_1}w_1,e^{i\theta_2}w_2,e^{i(\theta_1+\theta_2)}w_3) \quad \hbox{s.t.}\quad\theta_1, \theta_2 \in\R,  \vec{w}\in W^{+}_r\right\}\subset \mathcal G$.\\

We claim that for any $\vec{u}\in \mathcal G$, the Schwarz symmetric rearrangement of $(|u_{1}|,|u_{2}|, |u_{3}|)$, that we define by $\vec{w}:=(|u_{1}|^*, |u_{2}|^*, |u_{3}|^*)\in H^1_r(\R^N,\R^3)\cap S(a_1,a_2)$, belongs to $\mathcal{P}^+_{r,a_1,a_2}$. Indeed, if $ \sum^{3}_{j=1}\|\nabla w_j\|^2_2<\sum^{3}_{j=1}\|\nabla u_{j}\|^2_2$ or $\mathrm{Re}\int u_{1}u_{2}\overline{u}_3>\mathrm{Re}\int w_1w_2\overline{w}_3$, then $E(\vec{u})<E(\vec{w})$.
We have
\begin{equation*}
\begin{aligned}
\inf_{\vec{u}\in\mathcal{P}^{+}_{a_1,a_2}\cap \mathcal{M}} E(\vec{u})&=\inf_{\vec{u}\in S(a_1,a_2)\cap \mathcal{M}}\min _{0<\sigma\le s^+_{\vec{u}}} E(\sigma\star\vec{u})\\
&\le \min_{0<\sigma\le s^+_{\vec{w}}} E(\sigma\star\vec{w})< \min_{0<\sigma\le s^+_{\vec{w}}} E(\sigma\star\vec{u})=\inf_{\vec{u}\in\mathcal{P}^{+}_{a_1,a_2}\cap \mathcal{M}} E(\vec{u}),
\end{aligned}
\end{equation*}
which is a contradiction. In the chain of relations above, we used in order: the definition, the fact that $\vec w\in S(a_1,a_2)$, the relation $E(\vec{u})<E(\vec{w})$, and in the last identity we employed  the inequality  $s^{+}_{\vec u}\le s^{+}_{\vec w}$, jointly with the fact that $\vec u$ is in the set $\mathcal G$ of ground states. Therefore, $\vec{w}\in \mathcal{P}^{+}_{r,a_1,a_2}$ and $E(\vec{w})=E(\vec{u})$.
The fact that for any $\vec{u}\in \mathcal G$, we have $\vec{u}=(e^{i\theta_1}v_1,e^{i\theta_2}v_2,e^{i(\theta_1+\theta_2)}v_3)$ and $\vec{v}\in W^{+}_r$ is standard.
\end{proof}

\begin{Lem}\label{lem2.6}
Let $N\le 3$,  $2_*<p < 2^*$, and $\alpha,a_1, a_2>0$. If $\max\{a_1,a_2\}<D$, then \eqref{eqA0.2} has a ground state solution $(\lambda_1,\lambda_2,u_1,u_2,u_3)$ with $\lambda_1,\lambda_2>0$, and $\vec{u}\in S(a_1,a_2)$ is positive, radially symmetric, and decreasing.
\end{Lem}
\begin{proof}
By Lemma \ref{lem2.5}, we only need to show that $m^+_r(a_1,a_2)$ is attained.
Since $m^+_r(a_1,a_2)=\inf_{\vec u\in V(a_1,a_2)}E(\vec u)$,  and by using the symmetric decreasing rearrangement, we obtain a minimizing sequence $\{\vec{w}_n\}$,  where $\vec{w}_n\in H^1_r(\R^N,\R^3)\cap V(a_1,a_2)$ is decreasing for every $n$. Moreover, by Lemma \ref{lem2.4}, $E(s_{\vec{w}_n}\star \vec{w}_n)\le E(\vec{w}_n)$ and $s_{\vec{w}_n}\star \vec{w}_n\in V(a_1,a_2)$.
Replacing $\vec{w}_n$ by $s_{\vec{w}_n}\star \vec{w}_n$, we have a new minimizing sequence $s_{\vec{w}_n}\star \vec{w}_n\in \mathcal{P}^+_{a_1,a_2}\cap \mathcal{M}$. Thus, by the Ekeland's variational principle,
we can choose a non-negative radial Palais-Smale sequence $\{\vec{u}_n\}$ for $E|_{S(a_1,a_2)}$ at level $m^+_r(a_1,a_2)$, with $P(\vec{u}_n)=o_n(1)$ and such that $\lim\limits_{n\to \infty} E(\vec{u}_n)=m^+_r(a_1,a_2)$ and $E'|_{S(a_1,a_2)}(\vec{u}_n)\to 0$ as $n\to \infty$ (see also \cite[Lemma 3.7]{JL}). Since
\begin{equation*}
\begin{aligned}
m^+_r(a_1,a_2)+o_n(1)=E(\vec{u}_n)&=\left(\frac{1}{2}-\frac{1}{p\gamma_p}\right)\sum^{3}_{i=1}\|\nabla u_{i,n}\|^2_2-\left(1-\frac{N}{2p\gamma_p}\right)\alpha \int u_{1,n}u_{2,n}u_{3,n},\\
\end{aligned}
\end{equation*}
we have that the sequence $\{\vec{u}_n\}$ is bounded in $H^1_r(\R^N,\R^3)$. Indeed, using that $m^+_r(a_1,a_2)<0$, by  H\"older  and  Gagliardo-Nirenberg inequalities,
\begin{equation*}
\begin{aligned}
\sum^{3}_{i=1}\int |\nabla u_{i,n}|^2&\le \frac{(2p\gamma_p-N)\alpha}{p\gamma_p-2} \int u_{1,n}u_{2,n}u_{3,n}\\
& \le \frac{(2p\gamma_p-N)\alpha}{3(p\gamma_p-2)} \max\{a^{\frac{6-N}{2}}_1,a^{\frac{6-N}{2}}_2\}C(N,p)^3\left(\sum^{3}_{i=1}\|\nabla u_{i,n}\|^{\frac{N}{2}}_2\right).
\end{aligned}
\end{equation*}
As  $2_*<p<2^*$,  we have $\frac{N}{2}<2<p\gamma_p$, hence the boundedness. Then there exists $(u_1,u_2,u_3)$ such that $(u_{1,n},u_{2,n},u_{3,n})\rightharpoonup (u_1,u_2,u_3)$ weakly in $H^1_r(\R^N,\R^3)$, $(u_{1,n},u_{2,n},u_{3,n})\to (u_1,u_2,u_3)$ strongly in $L^{r}\times L^r\times L^r$ for $r\in (2, 2^*)$, and a.e. in $\R^N\times \R^N\times \R^N$ as $n\to \infty$. Therefore, $u_i\ge 0$ are radial functions for  $i=1,2,3$.
\vskip2mm
By the Lagrange multiplier's rule (see \cite[Lemma 3]{BL}), we know that there exists a sequence $\{(\lambda_{1,n},\lambda_{2,n})\}\subset \R\times\R$ such that
\begin{equation}\label{b7}
\begin{cases}
\displaystyle \int \left(\nabla u_{1,n}\nabla\phi_1+\lambda_{1,n}u_{1,n}\phi_1-|u_{1,n}|^{p-2}u_{1,n}\phi_1-\alpha u_{3,n}u_{2,n}\phi_1\right)=o_n(1)\|\phi_1\|_{H^1(\R^N)},\vspace{0.5ex}\\
\displaystyle\int \left(\nabla u_{2,n}\nabla\phi_2+\lambda_{2,n}u_{2,n}\phi_2-|u_{2,n}|^{p-2}u_{2,n}\phi_2-\alpha u_{3,n}u_{1,n}\phi_2\right)=o_n(1)\|\phi_2\|_{H^1(\R^N)},\vspace{0.5ex}\\
\displaystyle\int \left(\nabla u_{3,n}\nabla\phi_3+(\lambda_{1,n}+\lambda_{2,n})u_{3,n}\phi_3-|u_{3,n}|^{p-2}u_{3,n}\phi_3-\alpha u_{1,n}u_{2,n}\phi_2\right)=o_n(1)\|\phi_3\|_{H^1(\R^N)},
\end{cases}
\end{equation}
as $n\to \infty$, for every $\phi_i\in H^1(\R^N,\R)~(i=1,2,3)$. In particular, if we take $(\phi_1,\phi_2,\phi_3)=(u_{1,n},u_{2,n},u_{3,n})$, we have  that $(\lambda_{1,n},\lambda_{2,n})$ is bounded, therefore up to a subsequence we have convergence  $(\lambda_{1,n},\lambda_{2,n})\to (\lambda_{1},\lambda_{2})\in \R^2$.  Since $(u_{1,n},u_{2,n},u_{3,n})\rightharpoonup (u_1,u_2,u_3)$ weakly in $H^1_r(\R^N,\R^3)$, passing to the limit in \eqref{b7}, we see that  $(u_1,u_2,u_3)$ weakly solves
\begin{equation*}
\begin{cases}
-\Delta u_1+\lambda_1u_1=|u_1|^{p-2}u_1+\alpha u_3u_2,\\
-\Delta u_2+\lambda_2u_2=|u_2|^{p-2}u_2+\alpha u_3u_1,\\
-\Delta u_3+(\lambda_1+\lambda_2)u_3=|u_3|^{p-2}u_3+\alpha u_1u_2.
\end{cases}
\end{equation*}
In addition, we claim that $\displaystyle\mathrm{Re}\int u_{1}u_2\overline{u}_3>0$. If not, we have
\begin{equation*}
\sum^{3}_{i=1}\|\nabla u_i\|^2_2\le \gamma_p \sum^{3}_{i=1}\|u_i\|^p_p\le p\gamma_p A_1 \left(\sum^{3}_{i=1}\|\nabla u_i\|^2_2\right)^{\frac{p\gamma_p}{2}},
\end{equation*}
and then $\left(p\gamma_p A_1\right)^{-\frac{2}{p\gamma_p-2}}\le \sum^{3}_{i=1}\|\nabla u_i\|^2_2$. Moreover, as $\mathcal{P}^+_{a_1,a_2}\subset V(a_1,a_2)$, we get $\vec{u}\in B_{\rho^*}$, and this is a contradiction with $\max\{a_1,a_2\}<D$.
From $P(\vec{u})=0$, we conclude that
\begin{equation}\label{d3}
\lambda_1\|u_1\|^2_2+\lambda_2\|u_2\|^2_2+(\lambda_1+\lambda_2)\|u_3\|^2_2=\sum^{3}_{i=1}(1-\gamma_p)\|u_i\|^p_p
+\left(3-\frac{N}{2}\right)\alpha\int u_1u_2u_3.
\end{equation}
By $P(\vec{u}_n)=o_n(1)$, we obtain
\begin{equation}\label{d4}
\begin{aligned}
\lambda_1a^2_1+\lambda_2a^2_2&=\lim_{n\to\infty}\left(\lambda_1\|u_{1,n}\|^2_2+\lambda_2\|u_{2,n}\|^2_2+(\lambda_1+\lambda_2)\|u_{3,n}\|^2_2\right)\\
&=\lim_{n\to\infty}\left(\sum^{3}_{i=1}(1-\gamma_p)\|u_{i,n}\|^p_p+\left(3-\frac{N}{2}\right)\alpha\int u_{1,n}u_{2,n}u_{3,n}\right)\\
&=(1-\gamma_p)\|u_{i}\|^p_p+\left(3-\frac{N}{2}\right)\alpha\int u_{1}u_{2}u_{3}.
\end{aligned}
\end{equation}
\vskip1mm
\noindent We claim that $u_1\not\equiv0$, $u_2\not\equiv 0$ and $u_3\not\equiv0$.
\vskip1mm
\noindent {\bf Case 1.} If $u_i\equiv0$ for any $i=1,2,3$, then $\displaystyle\int |u_{i,n}|^p\to 0$, $\displaystyle\int u_{1,n}u_{2,n}u_{3,n}\to 0$, we have
\begin{equation*}
P(\vec{u}_n)=\sum^{3}_{i=1}\|\nabla u_{i,n}\|^2_2=o_n(1).
\end{equation*}
Therefore,
\begin{equation*}
m^+_r(a_1,a_2)+o_n(1)=E(\vec{u}_n)=o_n(1),
\end{equation*}
and this contradicts the fact that $m^+_r(a_1,a_2)<0$.
\vskip1mm
\noindent{\bf Case 2.} If $u_i\not\equiv0$, $u_j\equiv0$ and $u_l\equiv0$, $i,j,l\in \{1,2,3\}$, then $u_{j,n}\to 0$ and $u_{l,n}\to 0$ in $L^p$. Let $\tilde{u}_{i,n}=u_{i,n}-u_i$, then  $\tilde{u}_{i,n}\to 0$ in $L^p$.
By the Brezis-Lieb Lemma \cite{BrLi}, we deduce that
\begin{equation*}
\begin{aligned}
P(\vec{u}_n)&=\sum^{3}_{i=1}\|\nabla u_{i,n}\|^2_2-\gamma_p\|u_{i,n}\|^p_p+o_n(1)\\
&=\|\nabla \tilde{u}_{i,n}\|^2_2+\|\nabla u_{j,n}\|^2_2+\|\nabla u_{l,n}\|^2_2+\|\nabla u_i\|^2_2-\gamma_p\|u_i\|^p_p+o_n(1).
\end{aligned}
\end{equation*}
Thus,
\begin{equation*}
\begin{aligned}
m^+_r(a_1,a_2)+o_n(1)&=E(\vec{u}_n)=\left(\frac{\gamma_p}{2}-\frac{1}{p}\right)\|u_i\|^p_p+o_n(1)\ge 0,
\end{aligned}
\end{equation*}
which contradicts  $m^+_r(a_1,a_2)<0$.
\vskip1mm
\noindent{\bf Case 3.} If $u_i\not\equiv0$, $u_j\not\equiv 0$ and $u_l\equiv0$. By the structure of system \eqref{eqA0.2}, we get $u_i \equiv 0$ or $u_j\equiv0$, so Case 3 does not happen.

\vskip1mm

\noindent Therefore, $u_i\not\equiv0$ for all $i=1,2,3$. It remains to show that $m^+_r(a_1,a_2)$ is achieved.
From \cite[Lemma A.2]{IN}, we get $\lambda_1,\lambda_2>0$. Moreover, combining \eqref{d3} with \eqref{d4}, we have
\begin{equation}\label{d51}
\lambda_1a^2_1+ \lambda_2a^2_2=\lambda_1\|u_1\|^2_2+\lambda_2\|u_2\|^2_2+(\lambda_1+\lambda_2)\|u_3\|^2_2.
\end{equation}
Since $\|u_1\|^2_2+\|u_3\|^2_2\le a^2_1$ and $\|u_2\|^2_2+\|u_3\|^2_2\le a^2_2$, it follows from \eqref{d51} that $\|u_1\|^2_2+\|u_3\|^2_2=a^2_1$ and $\|u_2\|^2_2+\|u_3\|^2_2=a^2_2$, and hence $\vec{u}\in \mathcal{P}_{r,a_1,a_2}$. By the maximum principle (see \cite[Theorem 2.10]{hanq}), $u_i>0~(i=1,2,3)$. We then conclude that $\vec{u}_n\to \vec{u}$ in $H^1_r(\R^N,\R^3)$ and $E(\vec{u})=m^+_r(a_1,a_2)$.
\vskip2mm
\noindent In conclusion, we have proved that $m^+_r(a_1,a_2)$ is attained by a function $\vec{u}$ which is positive, radially symmetric, and decreasing in $r=|x|$. Therefore, the proof is complete.
\end{proof}

\vskip1mm
We look for the existence of $(\omega_1,\omega_2,\vec v)\in \R^2\times H^1(\R^N,\mathbb{C}^3)$ satisfying \eqref{n1}
 (see also \cite{KO,LO}) with $Q_1(\vec{v})=a^2_1$ and $Q_2(\vec{v})=a^2_2$.
It is important for our purpose to study the asymptotic behavior of minimizers for $m^+(a_1,a_2)$ because somehow \eqref{n1} can be seen as a limiting equation of problem \eqref{eqA0.2}, see Proposition \ref{lem4.3} below.
Then, we find the critical points of $E_0:H^1(\R^N,\mathbb{C}^3)\to \R$
\begin{equation*}\label{def:E0}
E_0(\vec{v}):=\frac{1}{2}\sum^{3}_{i=1}\|\nabla v_i\|^2_2-\mathrm{Re}\int v_1v_2\overline{v}_3
\end{equation*}
constrained on $S(a_1,a_2)$. Let us observe that  in \cite[Theorem 1.3]{KO}, only the case $N\le 2$ is considered. Define
\begin{equation}\label{m2}
0>m_0(a_1,a_2):=\inf_{\vec{v}\in S(a_1,a_2)}E_0(\vec{v})>-\infty.
\end{equation}

We have the following Lemmas.
\begin{Lem}\label{lem4.1}
Let $N\leq 3$. For any $a_1,a_2>0$
\[
m_0(a_1,a_2)=m_{0,r}(a_1,a_2):=\inf_{ \vec u\in S(a_1,a_2)\cap H^1_r(\R^N,\R^3)}E(\vec u).
\] In addition, $m_0(a_1,a_2)$ is reached by the vector function $(e^{i\theta_1}w_1, e^{i\theta_1}w_2, e^{i(\theta_1+\theta_2)}w_3 )$ where $E(\vec{w})=\inf\limits_{ \vec u\in S(a_1,a_2)\cap H^1_r(\R^N,\R^3)} E(\vec u)$, for some $(\theta_1,\theta_2)\in\R^2$.
\end{Lem}
\begin{proof}
It is easy to see that $m_0(a_1,a_2)\le m_{0,r}(a_1,a_2)$. Since $\|\nabla u_{i}\|^2_2\geq \|\nabla |u_{i}|^*\|^2_2$ it is also straightforward that  for any $\vec{u}\in S(a_1,a_2)$ one has $
E_0(\vec{u})\geq E_0(|u_1|^*,|u_2|^*,|u_3|^*)\geq m_{0,r}(a_1,a_2)$.
Therefore, $m_0(a_1,a_2)\ge m_{0,r}(a_1,a_2)$. Arguing as in the proof of Lemma \ref{lem2.5}, we obtain that  $m_0(a_1,a_2)$ is reached by the vector function $(e^{i\theta_1}w_1, e^{i\theta_1}w_2, e^{i(\theta_1+\theta_2)}w_3 )$ where $E_0(\vec{w})=\inf\limits_{\vec v\in S(a_1,a_2)\cap H^1_r(\R^N,\R^3)} E_0(\vec v)$ and $(\theta_1,\theta_2)\in\R^2$.
\end{proof}

\begin{Lem}\label{lem4.2}
Let $N\leq 3$. For any $a_1,a_2>0$,  $m_0(a_1,a_2)$ is reached by a real-valued, positive, radially symmetric, and decreasing function.
\end{Lem}
\begin{proof}
The proof follows the same lines on that of Lemma \ref{lem2.6}, arriving in this case to a solution $\vec v$ to 
\begin{equation}\label{y7}
\begin{cases}
-\Delta v_{1}+\omega_1v_{1}=v_{3}v_{2},\\
-\Delta v_{2}+\omega_2v_{2}=v_{3}v_{1}, \\
-\Delta v_{3}+(\omega_1+\omega_2)v_{3}=v_{1}v_{2}.
\end{cases}
\end{equation}
The Pohozaev-type identity for solutions of \eqref{y7} is given by
\[
P_0(\vec{v}):=\sum^3_{i=1}\|\nabla v_i\|^2_2-\frac{N}{2}\int v_1v_2v_3=0,
\]
then we have
\begin{equation}\label{y8}
\omega_1a^2_1+\omega_2a^2_2=\left(3-\frac{N}{2}\right)\int v_1v_2v_3
=\omega_1\left(\|v_1\|^2_2+\|v_3\|^2_2\right)+\omega_2\left(\|v_2\|^2_2+\|v_3\|^2_2\right).
\end{equation}

\noindent It remains to show that $v_1\not\equiv0$, $v_2\not\equiv0$ and $v_3\not\equiv0$.\vskip1mm
As in Lemma \ref{lem2.6}, we can separate the analysis in three cases and the proof is similar except for second case. In this case, either $-\Delta v_i +\omega_iv_i=0$ if $\omega_i>0$ or $-\Delta v_i \ge 0$ if $\omega_i\le 0$ (cf. \cite[Lemma A.2]{IN}), we obtain a contradiction to the assumption.

\vskip1mm
By the same argument as in the proof of Lemma \ref{lem2.6}, we have $\omega_1,\omega_2>0$. Then, by the strong maximal principle, $\vec{v}$ is a positive solution of \eqref{n1}.
It follows from \eqref{y8} that $\vec{v}\in S(a_1,a_2)$. Hence, $E_0(\vec{v})=m_0(a_1,a_2)$.
\end{proof}

\begin{Rem}\label{rem:KO-sol}\rm A straightforward modification of the proof of Lemma \ref{lem4.2} solves a problem left open in \cite{KO} in the case $N=3$. Indeed, instead of considering  the minimization  problem \eqref{m2}, we consider as in \cite{KO} the problem
\[
\Sigma_0(\gamma,\mu,s):=\inf\left\{E_0(\vec{u}) \quad  \hbox{s.t.} \quad \vec{u}\in H^1(\R^N,\mathbb{C}^3), \|u_1\|^2_2=\gamma, \|u_2\|^2_2=\mu, \|u_3\|^2_2 = \nu\right\},
\]
and a similar analysis as the one in Lemma \ref{lem4.2} gives a positive answer to \cite[Theorem 1.3 (ii)]{KO} in the three-dimensional case. In addition, if $2_*<p<2^*$, under  scaling transformation, $\alpha^{-\frac{N}{4-N}}\vec{u}\left(\alpha^{-\frac{2}{4-N}}x\right)\to \vec{v}$ in $H^1(\R^N,\mathbb{C}^3)$ as $\alpha \to 0$, where $\vec{u}\in \mathcal G$ and $\vec{v}$ is a ground state of \eqref{n1} on $S(a_1,a_2)$ (see Proposition \ref{lem4.3}).
\end{Rem}

In the following, we derive an improved upper bound of $m^+_r(a_1,a_2)$ when $a_1=a_2$. Indeed, we show in Lemma \ref{lem3.5} below, that $m^+_r(a_1,a_1)$ is not only negative, but bounded away from zero. Compare \eqref{m1} and \eqref{eq:m1bis}. We consider the problem
\begin{equation}\label{g1}
\begin{cases}
-\Delta u +\lambda u =\alpha u^2,\\
\displaystyle\int |u|^2=a^2,
\end{cases}
\end{equation}
where $\alpha,a>0$. Define
\begin{equation*}\label{def:J0}
J_0(u)=\frac{1}{2}\| \nabla u\|^2_{2}-\frac{\alpha}{3}\|u\|^{3}_{3},
\end{equation*}
then solutions $u$ of \eqref{g1} can be found, see \cite{Soave1}, as minimizers of
\begin{equation*}\label{e1}
0>m_0(a):=\inf_{u\in S(a)} J_0(u)>-\infty,
\end{equation*}
where $\lambda$ is a Lagrange multiplier, and $S(a):=\left\{u\in H^1(\R^N,\R) \quad \hbox{s.t.} \quad \|u\|^2_2=a^2\right\}$. From \cite{BJS}, we obtain that \eqref{g1} has a unique positive solution $(\lambda,u_{\alpha})$ given by
\begin{equation}\label{eq1.3}
\begin{aligned}
\lambda=\left(\frac{\alpha^2a^2}{\|w\|^2_2} \right)^{\frac{2}{4-N}},\quad  u_{\alpha}=\frac{\lambda}{\alpha}w(\lambda^{\frac{1}{2}}x), \quad -\Delta w+w-w^2=0,
\end{aligned}
\end{equation}
and we recall that $w$ is unique and positive.
We have
\begin{equation*}
m_0(a)=-\frac{4-N}{2(6-N)}\left(\frac{\alpha^2}{\|w\|^2_2}\right)^{\frac{2}{4-N}}a^{\frac{2(6-N)}{4-N}}<0.
\end{equation*}

\begin{Lem}\label{lem3.5}
Let $N\le 3$,  $2_*<p < 2^*$, and $\alpha,a_1, a_2>0$. If $a_1=a_2<D$, then
\begin{equation}\label{eq:m1bis}
m^+(a_1,a_1)< 3m_0\left(\frac{a_1}{\sqrt{2}}\right)<0.
\end{equation}
\end{Lem}
\begin{proof}
$m_0\left(\frac{a_1}{\sqrt{2}}\right)$ is achieved by $\tilde{u}_1\in S\left(\frac{a_1}{\sqrt{2}}\right)$ and $\tilde{u}_1$ is radially symmetric and decreasing, see \cite{CaLi}.
By adopting the same notation as in Lemma \ref{lem2.11}, we have
\begin{equation}\label{x1}
h(\rho)<h_1(\rho):=\frac{1}{2}\rho^{2}-\frac{\alpha}{3} C^3(N,p) a_1^{\frac{6-N}{2}}\rho^{\frac{N}{2}},
\end{equation}
where, by H\"older inequality, we have  that
\begin{equation*}
\begin{aligned}
\alpha \mathrm{Re}\int u_1u_2\overline{u}_3\le \alpha \|u_1\|_3\|u_2\|_3\|u_2\|_3\le \frac{\alpha}{3}\sum^{3}_{i=1}\|u_i\|^3_3.
\end{aligned}
\end{equation*}
Computations similar to those in \eqref{b3} give $J_0(\vec{u})\ge h_1\left(\left(\sum^{3}\limits_{i=1}\|\nabla u_i\|^2_2\right)^{\frac{1}{2}}\right)$.
By direct calculations, there exists $0<\hat{\rho}<R_0$ such that $h_1(\hat{\rho})=0$.
Then, we have
\begin{equation*}
3\|\nabla \tilde{u}_1\|^2_2\le \hat{\rho}^2<R^2_0< (\rho^*)^2.
\end{equation*}
Since $h(R_0)=h(R_1)=0$, by the monotonicity of $h(\rho)$, we deduce that $(\tilde{u}_1,\tilde{u}_1,\tilde{u}_1)\in V(a_1,a_1)$.
It implies that
\begin{equation*}
\begin{aligned}
m^+(a_1,a_1)=\inf_{\vec{u} \in V(a_1,a_1)} E(\vec{u})\le E(\tilde{u}_1,\tilde{u}_1,\tilde{u}_1)= 3J_0(\tilde{u}_1)-\frac{3}{p}\|\tilde{u}_1\|^p_p<3m_0\left(\frac{a_1}{\sqrt{2}}\right).
\end{aligned}
\end{equation*}
Hence, the proof is complete.
\end{proof}

\begin{Lem}\label{lem3.6}
Let $N\le 3$,  $2_*<p < 2^*$, and $\alpha,a_1, a_2>0$. If $a_1=a_2<D$, then for any ground state $\vec{u}\in S(a_1,a_1)$ of \eqref{eqA0.2},
for $a_1\to 0$ we have, up to a subsequence,
\begin{equation*}
\left(\alpha\kappa^{-1}u_1(\kappa^{-\frac{1}{2}}x), \alpha\kappa^{-1}u_2(\kappa^{-\frac{1}{2}}x),\alpha\kappa^{-1}u_3(\kappa^{-\frac{1}{2}}x)\right)\to \vec{v}_0
 \ \  \text{in}\ H^{1}(\R^N,\mathbb{C}^3),
\end{equation*}
where $\vec{v}_0$ is a ground state solution of $E_0$ constrained on $S(\sqrt{2}\|w\|_2,\sqrt{2}\|w\|_2)$, $w$ is defined in \eqref{eq1.3}, and $\kappa=\left(\frac{\alpha a_1}{\sqrt{2}\|w\|_2} \right)^{\frac{4}{4-N}}$.
\end{Lem}
\begin{proof}
Fix $\alpha>0$. For any $\{a_n\}\subset\R^+$ with $a_n\to 0^+$ as $n\to +\infty$,  let $\vec{u}_n\in V(a_n,a_n)$ be a minimizer of $m^+(a_n,a_n)$, where $V(a_n,a_n)=\left\{ \vec{u}_{n}\in S(a_n,a_n)\cap\mathcal{M} : \left(\sum^{3}_{i=1}\|\nabla u\|^2_{2}\right)^{\frac{1}{2}}< \rho^* \right\}$.
By Lemma \ref{lem2.6}, we get that $\vec{u}_{n}$ is a ground state of $E\vert_{S(a_n,a_n)}$. Then the Lagrange multipliers rule implies the existence of some $\lambda_{1,a_n},\lambda_{2,a_n} \in \R$ such that
\begin{equation}\label{y4}
\begin{cases}
\displaystyle\int\left( \nabla u_{1,n}\nabla\overline{\phi}_1+ \lambda_{1,a_n}u_{1,n}\overline{\phi}_1
-|u_{1,n}|^{p-2}u_{1,n}\overline{\phi}_1\right)=\alpha\mathrm{Re}\int u_{3,n}\overline{u}_{2,n}\overline{\phi}_1,\vspace{0.5ex}\\
\displaystyle\int \left(\nabla u_{2,n}\nabla \overline{\phi}_2+ \lambda_{2,a_n}u_{2,n}\overline{\phi}_2
-|u_{2,n}|^{p-2}u_{2,n}\overline{\phi}_2\right)=\alpha\mathrm{Re}\int u_{3,n}\overline{u}_{1,n}\overline{\phi}_2,\vspace{0.5ex}\\
\displaystyle\int\left( \nabla u_{3,n}\nabla \overline{\phi}_3+ (\lambda_{1,a_n}+\lambda_{2,a_n})u_{3,n}\overline{\phi}_3
-|u_{3,n}|^{p-2}u_{3,n}\overline{\phi}_3\right)=\alpha\mathrm{Re}\int u_{1,n}u_{2,n}\overline{\phi}_3,\\
\end{cases}
\end{equation}
for each $\vec{\phi}\in H^{1}(\R^N,\mathbb{C}^3)$.
\vskip1mm
\noindent We claim that
\begin{equation}\label{Dd2.2}
\frac{1-\gamma_p}{\gamma_p}\left(\frac{\alpha2N(p-3)}{N(p-2)-4}C^{3}(N,p)\right)^{\frac{4}{4-N}}a^{\frac{4}{4-N}}_n>\lambda_{1,a_n}+\lambda_{2,a_n}>6K_{N}
a^{\frac{4}{4-N}}_n,
\end{equation}
where $K_N:=\frac{4-N}{4(6-N)}\left(\frac{\alpha^2}{2\|w\|^2_2}\right)^{\frac{2}{4-N}}$.
Indeed, it follows from \eqref{y4} that
\begin{equation*}
\begin{aligned}
(\lambda_{1,a_n}+\lambda_{2,a_n})a^2_n=-\sum^{3}_{i=1}\|\nabla u_{i,n}\|^2_{2}+\sum^{3}_{i=1}\|u_{i,n}\|^{p}_{p}+3\alpha \mathrm{Re}\int u_{1,n}u_{2,n}\overline{u}_{3,n}> 6K_N a^{\frac{2(6-N)}{4-N}}_n.
\end{aligned}
\end{equation*}

\noindent Since $P(\vec{u}_{n})=0$, by Lemma \ref{lem3.5} we have
\begin{equation}\label{x3}
\begin{aligned}
E(\vec{u}_{n})&=\left(\frac{1}{2}-\frac{1}{p\gamma_p}\right)\sum^3_{i=1}\|\nabla u_{i,n}\|^2_2-\frac{\alpha(p-3)}{p-2}\mathrm{Re}\int u_{1,n}u_{2,n}u_{3,n}\\
&=-\frac{4-N}{2N}\sum^3_{i=1}\|\nabla u_{i,n}\|^2_2+\gamma_p\left(\frac{2}{N}-\frac{1}{p\gamma_p}\right)\sum^3_{i=1}\| u_{i,n}\|^p_p\\
&< -3K_{N}a^{\frac{2(6-N)}{4-N}}_n.
\end{aligned}
\end{equation}
It follows immediately that
\begin{equation}\label{ba1}
\frac{6N}{4-N}K_{N}a^{\frac{2(6-N)}{4-N}}_n< \sum^3_{i=1}\|\nabla u_{i,n}\|^2_2<
\left(\frac{\alpha2N(p-3)C^{3}(N,p)}{N(p-2)-4}\right)^{\frac{4}{4-N}}
a^{\frac{2(6-N)}{4-N}}_n.
\end{equation}
Hence, combining it with $P(\vec{u}_{n})=0$, we obtain that
\begin{equation*}
\begin{aligned}
(\lambda_{1,a_n}+\lambda_{2,a_n})a^2_n &=\left(\frac{1}{\gamma_p}-1\right)\sum^{3}_{i=1}\|\nabla u_{i,n}\|^2_{2}-\left(3-\frac{N}{2\gamma_p}\right)\alpha\mathrm{Re}\int u_{1,n}u_{2,n}u_{3,n}\\
&<\frac{1-\gamma_p}{\gamma_p}\left(\frac{\alpha2N(p-3)C^{3}(N,p)}{N(p-2)-4}\right)^{\frac{4}{4-N}}a^{\frac{2(6-N)}{4-N}}_n.
\end{aligned}
\end{equation*}
The proof of  \eqref{Dd2.2} is complete.

\noindent Define now
\begin{equation}\label{def:v-i}
v_{1,n}:=\alpha\kappa_n^{-1}u_{1,n}(\kappa_n^{-\frac{1}{2}}x), \  \
v_{2,n}:=\alpha\kappa_n^{-1}u_{2,n}(\kappa_n^{-\frac{1}{2}}x),\quad \text{and} \quad
v_{3,n}:=\alpha\kappa_n^{-1}u_{3,n}(\kappa_n^{-\frac{1}{2}}x),
\end{equation}
where $\kappa_n=\left(\frac{\alpha a_n}{\sqrt{2}\|w\|_2} \right)^{\frac{4}{4-N}}$. Then, for $i=1,2,3$,
\begin{equation*}
\|\nabla v_{i,n}\|^2_2=\kappa_n^{\frac{N-6}{2}}\alpha^{2}\|\nabla u_{i,n}\|^2_2, \quad \|v_{i,n}\|^p_p=\kappa_n^{\frac{N-2p}{2}}\alpha^{p}\|u_{i,n}\|^p_p, \quad \|v_{i,n}\|^2_2=\frac{2\|w\|^2_2}{a^2_n}\|u_{i,n}\|^2_2.
\end{equation*}
Therefore, for $a_n\to 0$ as $n\to \infty$, we have
\begin{equation*}
\begin{aligned}
m^+(a_n,a_n)+o_n(1)&= E(\vec{u}_n)= \kappa_n^{\frac{6-N}{2}}\alpha^{-2}E_0(\vec{v}_n)- \kappa_n^{\frac{2p-N}{2}}\alpha^{-(p-2)}\sum^{3}_{i=1} \|v_{i,n}\|^p_p\\
&\ge \kappa_n^{\frac{6-N}{2}}\alpha^{-2}m_0\left(\sqrt{2}\|w\|_2,\sqrt{2}\|w\|_2\right)+o\left(a^{\frac{2(6-N)}{4-N}}_n\right),
\end{aligned}
\end{equation*}
where we used  the definition of $\kappa_n$ to have
$\kappa_n^{\frac{2p-N}{2}} \sim  a_n^{\frac{2(2p-N)}{4-N}} $, then we can estimate the remainder with $o\left(a_n^{\frac{2(6-N)}{4-N}}\right) $, as $p>3$.

\noindent From the definition of $m_0(\sqrt{2}\|w\|_2,\sqrt{2}\|w\|_2)$, for any $\varepsilon>0$, there exists $\vec{v}_0\in S(a_1,a_2)$ such that
\begin{equation*}
E_0(\vec{v}_0)\le m_0\left(\sqrt{2}\|w\|_2,\sqrt{2}\|w\|_2\right)+\varepsilon.
\end{equation*}
Let $u_{i,a_n}:=\kappa_n\alpha^{-1}v_{i,0}(\kappa_n^{1/2}x)$ for $i=1,2,3$. Therefore, $\vec{u}_{a_n}\in V(a_n,a_n)$ for $a_n$ small enough. Then
\begin{equation*}\label{x21}
\begin{aligned}
m^+(a_n,a_n)&=\inf_{\vec{u} \in V(a_n,a_n)} E(\vec{u})\le E(u_{1,a_n},u_{2,a_n},u_{3,a_n})\\
&\le  \kappa_n^{\frac{6-N}{2}}\alpha^{-2}E_0(\vec{v}_0)+ \kappa_n^{\frac{2p-N}{2}}\alpha^{-p}\sum^{3}_{i=1}\|v_{i,0}\|^p_p\\
&\le   \kappa_n^{\frac{6-N}{2}}\alpha^{-2} \left(m_0(\sqrt{2}\|w\|_2,\sqrt{2}\|w\|_2)+\varepsilon\right)+o\left(a^{\frac{2(6-N)}{4-N}}_n\right).
\end{aligned}
\end{equation*}
for all $\varepsilon>0$ and $a_n>0$ small enough. Therefore,
\[
m^+(a_n,a_n)= \kappa_n^{\frac{6-N}{2}}\alpha^{-2}m_0(\sqrt{2}\|w\|_2,\sqrt{2}\|w\|_2)+o\left(a^{\frac{2(6-N)}{4-N}}_n\right).
\]
This implies that $\{\vec{v}_n\}$ is a minimizing sequence for $m_0\left(\sqrt{2}\|w\|_2,\sqrt{2}\|w\|_2\right)$. If $\{u_n\}$ is a minimizing sequence of $m^+(a_n,a_n)$,  $E(\vec{u}_n)=m^+(a_n,a_n)+o(1)$.
By the definition of $\{\vec{v}_n\}$, see \eqref{def:v-i}, we have
\[
E(\vec{v}_n)=E(\alpha \kappa_n^{-1}\vec{u}_n( \kappa_n^{-\frac{1}{2}}x))=m_0(\sqrt{2}\|w\|_2,\sqrt{2}\|w\|_2)+o(a^{\frac{2(6-n)}{4-n}}_n),
\]
i.e.,  $\{v_n\}$ is a minimizing sequence of $m_2(\sqrt{2}\|w\|_2,\sqrt{2}\|w\|_2)$.
Up to a subsequence, there exists a radially symmetric Palais-Smale sequence $\{\vec{\tilde{v}}_n\}$ such that $\|\vec{\tilde{v}}_n-\vec{v}_n\|_{H^1(\R^N,\mathbb{C}^3)}=o_n(1)$.
Similar to the proof of Lemma \ref{lem4.2}, up to translation, there exists a minimizer $\vec{v}_0$ for $m_0(\sqrt{2}\|w\|_2,\sqrt{2}\|w\|_2)$ such that $\vec{\tilde{v}}_n\to \vec{v}_0$ in $H^1(\R^N,\mathbb{C}^3)$.
Indeed, by Lemma \ref{lem4.2}
for any minimizing sequence of $m_0(\sqrt{2}\|w\|_2, \sqrt{2}\|w\|_2)$ , there exists a compact subsequence.
\end{proof}
\vskip2mm
\begin{Prop}\label{lem4.3}
Let $N\leq 3$, $2_*<p<2^*$,  $a_1, a_2>0$,  and suppose that $\max\{a_1,a_2\}<D$. Let $\{\alpha_n\}$ be a positive sequence with $\alpha_n\to 0$ as $n\to \infty$, and let $\vec{u}_n$ be a minimizer for $m^+(a_1,a_2)$ (with $\alpha=\alpha_n>0$), up to a subsequence,
\begin{equation*}
\vec{v}_n:=\alpha_n^{-\frac{N}{4-N}}\vec{u}_n\left(\alpha_n^{-\frac{2}{4-N}}x\right)\to \vec{v} \quad \text{in} \quad  H^1(\R^N,\mathbb{C}^3),
\end{equation*}
where $\vec{v}$ is a minimizer of $m_0(a_1,a_2)$.
\end{Prop}
\begin{proof}
Let $\{\alpha_n\}\subset (0,\infty)$ with $\alpha_n\to 0$. From the definition of $m_0(a_1,a_2)$, for any $\varepsilon>0$ sufficiently small, there exists $\vec{v}_0\in S(a_1,a_2)$ such that $E_0(\vec{v}_0)\le m_0(a_1,a_2)+\varepsilon$. Let $u_{i,\alpha_n}(x):=\alpha_n^{\frac{N}{4-N}}v_{i,0}\left(\alpha_n^{\frac{2}{4-N}}x\right),(i=1,2,3)$. As the calculation in \eqref{x1}, we have
$(u_{1,\alpha_n},u_{2,\alpha_n},u_{3,\alpha_n})\in V(a_1,a_2)$, and then
\begin{equation}\label{x2}
\begin{aligned}
m^+(a_1,a_2)&=\inf_{\vec{u} \in V(a_1,a_2)} E(\vec{u})\le E(u_{1,\alpha_n},u_{2,\alpha_n},u_{3,\alpha_n})\\
&\le  \alpha_n^{\frac{4}{4-N}}E_0(\vec{v}_0)+\alpha_n^{\frac{N(p-2)}{4-N}}\sum^{3}_{i=1}\|v_{i,0}\|^p_p\le  \alpha_n^{\frac{4}{4-N}} \left(m_0(a_1,a_2)+\varepsilon\right)+o(\alpha_n^{\frac{4}{4-N}}),
\end{aligned}
\end{equation}
for all $\varepsilon>0$ and $\alpha_n>0$ small enough.
\vskip1mm
Let $\vec{u}_{n}\in V(a_1,a_2)$ be a minimizer of $m^+(a_1,a_2)$ for $\alpha_n>0$. Then, combining \eqref{x2} and the same argument as in \eqref{x3} and \eqref{ba1}, we can prove that there exist $C_1,C_2>0$ such that $C_1\alpha_n^{\frac{4}{4-N}} \le \sum^3_{i=1}\|\nabla u_{i,n}\|^2_2 \le C_2 \alpha_n^{\frac{4}{4-N}}$.
Define
\begin{equation*}
\vec{v}_n:=\alpha_n^{-\frac{N}{4-N}}\vec{u}_n\left(\alpha_n^{-\frac{2}{4-N}}x\right).
\end{equation*}
Then, $\vec{v}_n\in S(a_1,a_2)$, and there exists $C>0$ such that for all $\alpha_n<1$, $\sum^3_{i=1}\|\nabla v_{i,n}\|^2_2\le C$.
Hence,
\begin{equation*}
\begin{aligned}
m^+(a_1,a_2)+o_n(1)&= E(\vec{u}_n)
=\alpha_n^{\frac{4}{4-N}}\left(E_0(\vec{v}_n)-\frac{\alpha_n^{\frac{N(p-2)-4}{4-N}}}{p}\sum^3_{i=1}\| v_{i,n}\|^p_p\right)\\
&\ge m_0(a_1,a_2)\alpha_n^{\frac{4}{4-N}}+o\left(\alpha_n^{\frac{4}{4-N}}\right).
\end{aligned}
\end{equation*}
Thus, it follows that
\begin{equation*}
m^+(a_1,a_2)=m_0(a_1,a_2)\alpha_n^{\frac{4}{4-N}}+o\left(\alpha_n^{\frac{4}{4-N}}\right).
\end{equation*}
This implies that $\{\vec{v}_n\}$ is a minimizing sequence for $m_0(a_1,a_2)$. Up to a subsequence, there exists a radially symmetric  Palais-Smale sequence $\{\vec{\tilde{v}}_n\}$ such that $\|\vec{\tilde{v}}_n-\vec{v}_n\|_{H^1(\R^N,\mathbb{C}^3)}=o_n(1)$.
Similar to Lemma \ref{lem4.2}, there exists a minimizer $\vec{v}$ for $m_0(a_1,a_2)$ such that
$\vec{\tilde{v}}_n\to \vec{v}$ in $H^1(\R^N,\mathbb{C}^3)$.
\end{proof}

\vskip2mm

\subsection{Mass-critical case}
In this subsection, we deal with the mass critical case $p=2_*=2+\frac{4}{N}$. As in the previous sections, $\alpha,a_1, a_2$ are positive. We recall the decomposition of $\mathcal{ P}_{a_1,a_2}=\mathcal{ P}_{a_1,a_2}^+\cup \mathcal{ P}_{a_1,a_2}^0\cup \mathcal{ P}_{a_1,a_2}^-$ as  in Section 2, see \eqref{eq:c41}. From the definition of $\mathcal{ P}_{a_1,a_2}^0$, i.e., $\Psi'_{\vec{u}}(1)=\Psi''_{\vec{u}}(1)=0$, then necessarily $u_i=0~(i=1,2,3)$. Therefore, $\mathcal{ P}_{a_1,a_2}^0=\emptyset$. Similarly  to Lemma \ref{lem2.2}, we can also claim that $\mathcal{ P}_{a_1,a_2}\cap \mathcal{M}$ is a smooth manifold of codimension 1 in $H^1(\R^N,\mathbb{C}^3)$.
\vskip1mm
\begin{Lem}\label{lem5.1}
If $\max\{a_1,a_2\}<\left(\frac{N+2}{N}\right)^{\frac{N}{4}}\left(C(N,2_*)\right)^{-\frac{N+2}{2}}$ and $a_1,a_2>0$, then for all $\vec{u}\in S(a_1,a_2)\cap\mathcal{M}$, there exists $\sigma_{\vec{u}}$ such that $\sigma_{\vec{u}}\star\vec{u}\in\mathcal{P}_{a_1,a_2}$. Further, $\sigma_{\vec{u}}$ is the unique critical point of the function $\Psi_{\vec{u}}$ and it is a strict minimum point at negative level.
Moreover:\smallskip

\noindent\textup{(i)} $\Psi_{\vec{u}}$ is strictly decreasing in $(0,\sigma_{\vec{u}})$,\smallskip

\noindent\textup{(ii)} $\mathcal{ P}_{a_1,a_2}=\mathcal{ P}_{a_1,a_2}^+$ and $P(\vec{u})<0$ if and only if $\sigma_{\vec{u}}<1$, \smallskip

\noindent \textup{(iii)} the map $\vec{u} \mapsto \sigma_{\vec{u}} \in \mathbb{R}^+$ is of class $ C^1$.
\end{Lem}
\begin{proof}
\noindent \textup{(i)} Using that $p=2_*$ and the   definition of $\Psi$ in \eqref{eq:fiber}, we have
\begin{equation}\label{y1}
\begin{aligned}
\Psi_{\vec{u}}(s) \ge \frac{s^2}{2}\left(\sum^{3}_{i=1}\|\nabla u_i\|^2_2\left(1-\frac{N\left(C(N,2_*)\right)^{2+\frac{4}{N}}}{N+2} \max\left\{a_1^{\frac{4}{N}},a_2^{\frac{4}{N}}\right\}\right)\right)-s^{\frac{N}{2}}\alpha \mathrm{Re}\int u_1u_2\overline{u}_3.
\end{aligned}
\end{equation}
Note that for any $\vec{u}\in \mathcal{M}$, $s\star \vec{u}\in \mathcal{P}_{a_1,a_2}$ if and only if $\Psi'_{\vec{u}}(s)=0$. From the latter property, if $1-\frac{N\left(C(N,2_*)\right)^{2+\frac{4}{N}}}{N+2}  \max\{a_1^{\frac{4}{N}},a_2^{\frac{4}{N}}\}$ is positive,
then $\Psi_{\vec{u}}(s)$ has a unique critical point $\sigma_{\vec{u}}$, which is a strict minimum point at negative level. Therefore, under the bound condition on  $\max\{a_1,a_2\}$ as in the statement of the Lemma, we have that
\[
\sum^{3}\limits_{i=1}\left(\frac{1}{2}\|\nabla u_i\|^2_2-\frac{N}{2N+4}\|u_i\|^{2_*}_{2_*}\right)>0.
\]
\vskip2mm

\noindent \textup{(ii)}  If $\vec{u}\in \mathcal{ P}_{a_1,a_2}\cap \mathcal{M}$, then $\sigma_{\vec{u}}$ is a minimum point,
we have that $\Psi''_{\vec{u}}(\sigma_{\vec{u}})\ge 0$. Since $\mathcal{P}^0_{a_1,a_2}=\emptyset$, we have $\vec{u}\in \mathcal{ P}^+_{a_1,a_2}$.  Finally, $\Psi'_{\vec{u}}(s)>0$ if and only if $s>\sigma_{\vec{u}}$, then $P(\vec{u})=\Psi'_{\vec{u}}(1)<0$ if and only if $\sigma_{\vec{u}}<1$.
\smallskip

\noindent \textup{(iii)} To prove that the map $\vec{u}\in S(a_1,a_2)\cap \mathcal{M}\mapsto \sigma_{\vec{u}}\in \R^+$ is of class $ C^1$, we can apply the Implicit Function Theorem as in Lemma \ref{lem2.3}.
\end{proof}

\begin{Lem} \label{cor5.1}
Let $N\leq 3$, assume $p=2_*$, and let $\alpha,a_1, a_2>0$. We have the followings:\smallskip

\noindent \textup{(i)} if $\max\{a_1,a_2\}<\left(\frac{N+2}{N}\right)^{\frac{N}{4}}\left(C(N,2_*)\right)^{-\frac{N+2}{2}}$, then
\begin{equation*}
-\infty<m^+(a_1,a_2):=\inf_{\vec{u}\in \mathcal{ P}^+_{a_1,a_2}\cap \mathcal{M}}E(\vec{u})=\inf_{\vec{u}\in S(a_1,a_2)}E(\vec{u})<0,
\end{equation*}
\smallskip

\noindent \textup{(ii)} if $\min\{a_1,a_2\}\ge\left(\frac{N+2}{N}\right)^{\frac{N}{4}}\left(C(N,2_*)\right)^{-\frac{N+2}{2}}$, then
%\begin{equation*}
$\inf_{\vec{u}\in S(a_1,a_2)}E(\vec{u})=-\infty$.
%\end{equation*}
\end{Lem}

\begin{proof}
We sketch the proof. As for
\textup{(i)}, from  \eqref{y1} we directly have that $E$ is coercive on  $S(a_1,a_2)$ and  $m^+(a_1,a_2)>-\infty$. By direct computations we have that $E(s\star \vec{u}) < 0$ for every $u\in \mathcal{P}^{+}_{a_1,a_2}\cap \mathcal{M}$ with $s>0$ small enough. Therefore, we know that $m^+(a_1,a_2) < 0$. \smallskip 
As for \textup{(ii)}, as in \cite[Section 3]{Soave1},  we can claim the existence of $\vec{u}\in S(a_1,a_2)$ such that $E(\vec{u})\le 0$, hence by \eqref{y1} and taking the limit  $\inf\limits_{\vec u\in S(a_1,a_2)}E(\vec u)=-\infty$.
\end{proof}

We state the following Lemmas, whose proofs are similar to the ones for Lemmas \ref{lem2.5}, \ref{lem4.1}, and Lemma \ref{lem3.6}, respectively.
\begin{Lem}\label{lem5.11}
Let $N\leq 3$. For $p=2_*$, $m^{+}(a_1,a_2)=m^+_r(a_1,a_2)$, where $m^+_r(a_1,a_2)$ is given by \eqref{y2}. In addition, $\inf\limits_{\mathcal{P}^{+}_{a_1,a_2}\cap \mathcal{M}} E(\vec{u})$ is attained  by $(e^{i\theta_1}w_1, e^{i\theta_1}w_2, e^{i(\theta_1+\theta_2)}w_3 )$ where $E(\vec{w})=\inf\limits_{\vec{u}\in \mathcal{P}^{+}_{r,a_1,a_2}\cap \mathcal{M}} E(\vec{u})$ and $(\theta_1,\theta_2)\in \R^2$.
\end{Lem}

\begin{Lem}\label{lem5.2}
Let $N\leq 3$, $\alpha,a_1, a_2>0$ and assume $p=2_*$.  If it holds that $\max\{a_1,a_2\}<\left(\frac{N+2}{N}\right)^{\frac{N}{4}}\left(C(N,2_*)\right)^{-\frac{N+2}{2}}$, then
$E|_{S(a_1,a_2)}$ has a critical point $\vec{u}$ at $m^+(a_1,a_2)$, and $\vec{u}$ is real-valued,  positive, and radially symmetric for some $\lambda_1,\lambda_2>0$.
\end{Lem}

\begin{Lem}\label{lem5.3}
Let $N\leq 3$, assume $p=2_*$, and let $\alpha>0$ and $a_1=a_2>0$.  If  $0<a_1<\left(\frac{N+2}{N}\right)^{\frac{N}{4}}\left(C(N,2_*)\right)^{-\frac{N+2}{2}}$, then for any ground state $\vec{u}\in S(a_1,a_1)$ of \eqref{eqA0.2}, if we let $a_1\to 0$, then 
we have
\begin{equation*}
\left(\alpha \kappa^{-1}u_1(\kappa^{-\frac{1}{2}}x), \alpha \kappa^{-1}u_2(\kappa^{-\frac{1}{2}}x),\alpha \kappa^{-1}u_3(\kappa^{-\frac{1}{2}}x)\right)\to \vec{v}_0
 \ \  \text{in}\ H^{1}(\R^N,\mathbb{C}^3),
\end{equation*}
where $\vec{v}_0$ is a minimizer of $m_0(\sqrt{2}\|w\|_2,\sqrt{2}\|w\|_2)$, $w$ is given in \eqref{eq1.3}, and $\kappa=\left(\frac{\alpha a_1}{\sqrt{2}\|w\|_2} \right)^{\frac{4}{4-N}}$.
\end{Lem}

We are now in position to prove the first main result of the paper, namely Theorem \ref{th1.1}.

\subsection{Proof of Theorem \ref{th1.1}.} We start with the intercritical case $2_*<p<2^*$. \smallskip

\noindent \textup{(i)} It follows from Lemmas \ref{lem2.5} and \ref{lem2.6} that there is a local minimizer of $E$ on $V(a_1,a_2)$.  \smallskip

\noindent \textup{(ii)} We shall prove that the set $\mathcal G$ defined in the Introduction is orbitally stable.  By contradiction,
suppose that there exist $\varepsilon_0>0$, a sequence of times $\{t_n\}\subset \R^+$, and a sequence of initial data $\{\vec{\psi}_{0,n}\} \subset H^1(\R^N,\mathbb{C}^3)$ such that the unique (for $n$ fixed) solution $\vec{\psi}_{\psi_{0,n}}(t)$ to the problem \eqref{eqA0.1} with initial datum $\vec{\psi}_{\psi_{0,n}}(0)=\vec{\psi}_{0,n}$ satisfies
\begin{equation*}
\text{dist}_{H^{1}(\R^N,\mathbb{C}^3)}\left(\vec{\psi}_{0,n},\mathcal G\right) < \frac{1}{n} \ \ \text{and} \ \  \text{dist}_{H^{1}(\R^N,\mathbb{C}^3)}\left(\vec{\psi}_{\psi_{0,n}}(t_n),\mathcal G\right)\ge \varepsilon_0.
\end{equation*}
Without loss of generality, we assume $\vec{\psi}_{0,n}\in S(a_1,a_2)$. Denote $\vec{\psi}_{\psi_{0,n}}(t_n)$ by $\vec{u}_n$.
Then by the conservation laws \eqref{eq:energy} and \eqref{eq:cons-masses}, $\{\vec u_n\} \subset H^1(\R^N,\mathbb{C}^3)$ satisfies $Q_1(\vec{u}_n)=a^2_1$, $Q_2(\vec{u}_n)=a^2_2$ and $E(\vec{u}_n) \to m^{+}(a_1,a_2)$.

\vskip1mm

\noindent We shall prove that for any $n\in \mathbb{N}$, $\vec{\psi}_{\psi_{0,n}}(t)$ is globally defined in time and $\vec{\psi}_{\psi_{0,n}}(t)\in B_{\rho^*}$ for any $t>0$, recalling that $\rho^*$ is given in Lemma \ref{lem2.11}. Since $\vec{\psi}_{0,n}\in B_{\rho^*}$, if $\vec{\psi}_{\psi_{0,n}}(t)$ leaves the set $B_{\rho^*}$, there exists $t_1\in (0,T_{\max})$ such that $\vec{\psi}_{\psi_{0,n}}(t_1)\in \partial B_{\rho^*}$, where $T_{\max}$ is the maximal forward time of existence for the solution $\vec{\psi}_{\psi_{0,n}}$. By \eqref{b3}, we have $E(\vec{\psi}_{\psi_{0,n}}(t_1))\ge h(\rho^*)\ge 0$, contradicting  the conservation of the energy.
If  $T_{\max}<\infty$,  by the blow-up alternative $\lim_{t\to T_{\max}^-}\left(\sum_{i=1}^{3}\|\psi_{\psi_{0,n}, i}(t)\|^2_2 \right)$ $=\infty$, then  there also exists $t_2\in (0,T_{\max})$ such that $\vec{\psi}_{\psi_{0,n}}(t_2)\in \partial B_{\rho^*}$. Analogously to the proof of the fact that $\vec{\psi}_{\psi_{0,n}}(t)\in B_{\rho^*}$, one shows that $T_{\max}=+\infty.$
This implies that solutions starting in $B_{\rho^*}$ are globally defined in time.
By Lemmas \ref{lem2.11} and \ref{lem2.4}, if $\max\{a_1,a_2\}<D$, we thus get
\begin{equation*}
\begin{aligned}
m^+(a_1,a_2)&=\inf_{\vec{u}\in S(a_1,a_2)\cap B_{\rho^*}\cap \mathcal{M}}E(\vec{u})\\
&=\inf\left\{E(\vec{u}) \hbox{ s.t. }\vec{u}\in S(a_1,a_2)\cap B_{\max\{a_1,a_2\}D^{-1}\rho^*}\cap\mathcal{M}\right\}.    
\end{aligned}
\end{equation*}
A similar analysis to that in the proof of \cite[Theorem 1.2]{KO} and \cite[Theorem 1.4]{Soave1},
yields strict sub-additivity of $E$ on $V(a_1,a_2)=S(a_1,a_2)\cap B_{\rho^*}\cap \mathcal{M}$. Moreover, combining $m^+(a_1,a_2)<0$ with $E(\vec{u}_n)\to m^+(a_1,a_2)$, we have that $\vec{u}_n\in \mathcal{M}$.
Therefore, there exists $\vec{u}\in \mathcal G$ such that $\vec{u}_n\to \vec{u}$ in $H^{1}(\R^N,\mathbb{C}^3)$. Since the set of ground states $\mathcal G$ is  invariant under translations, and this contradicts  $\text{dist}_{H^{1}(\R^N,\mathbb{C}^3)}\left(\vec{u}_n,\mathcal G\right)\ge\varepsilon_0 >0$.\smallskip

\noindent \textup{(iii)} The third point of the Theorem follows from Lemma \ref{lem3.6}.\smallskip

\noindent \textup{(iv)} The last point follows from \eqref{ba1}.

\medskip

We turn now the attention to the mass-critical case. For $p=2_*$ we have that \textup{(i)}, i.e., existence of minimizer of $m^+(a_1,a_2)$, follows from  Lemmas \ref{lem5.11} and \ref{lem5.2};  the orbital stability of $\mathcal G$ as in \textup{(ii)}  can be proved following  \cite[Theorem 1.4]{AH} or \cite[Theorem 1.2]{KO};  \textup{(iii)} follows from Lemma \ref{lem5.3}.

\noindent We conclude with the proof of \textup{(iv)}. By recalling that the ground state has a negative energy, by using the estimate in \eqref{b3} with $p=2_*$ we obtain
\begin{equation*}
\begin{aligned}
0>E(\vec{u})&\ge \frac{1}{2}\sum^{3}_{i=1}\|\nabla u_i\|^2_2\left(1-\frac{N\left(C(N,2_*)\right)^{2+\frac{4}{N}}}{N+2} \max\left\{a_1^{\frac{4}{N}},a_2^{\frac{4}{N}}\right\}\right)\\
&  \quad-\frac{\alpha C^{3}(N,2_*)}{3}\max\left\{a^{\frac{6-N}{2}}_1,a^{\frac{6-N}{2}}_2\right\}\left(\sum^{3}_{i=1}\|\nabla u_i\|^2_2\right)^{\frac{N}{4}},\\
\end{aligned}
\end{equation*}
then, if $\alpha\to 0$, we have
\begin{equation*}
\begin{aligned}
& \frac{1}{2}\left(\sum^{3}_{i=1}\|\nabla u_i\|^2_2\right)^{1-\frac{N}{4}}\left(1-\frac{N\left(C(N,2_*)\right)^{2+\frac{4}{N}}}{N+2} \max\left\{a_1^{\frac{4}{N}},a_2^{\frac{4}{N}}\right\}\right)\\&<\frac{\alpha C^{3}(N,2_*)}{3}\max\left\{a^{\frac{6-N}{2}}_1,a^{\frac{6-N}{2}}_2\right\}\to 0.
\end{aligned}
\end{equation*}
The proof is complete.

\vskip1mm
\section{Proof of Theorem \ref{th1.2}}
In this section, for $\alpha,a_1, a_2>0$, $2_*<p<2^*$ for $N\le 3$, we study the existence and properties of the second standing wave solution of \eqref{eqA0.2}.
Define
\begin{equation*}
m^{-}(a_1,a_2)=\inf_{\vec{u}\in \mathcal{P}^{-}_{a_1,a_2}} E(\vec{u}).
\end{equation*}
 By Lemma \ref{lem2.2}, if $\max\{a_1,a_2\}<D$, we check that $\mathcal{P}^0_{a_1,a_2}$ is empty.
Similar to the proof of Lemma \ref{lem2.5}, we get that if $\max\{a_1,a_2\}<D$, then
$\inf_{\vec{u}\in \mathcal{P}^{-}_{r,a_1,a_2}} E(\vec{u})=\inf_{\vec{u}\in\mathcal{P}^{-}_{a_1,a_2}} E(\vec{u})$.
Furthermore, $\inf\limits_{\mathcal{P}^{-}_{a_1,a_2}} E(\vec{u})$ is reached by the vector function $(e^{i\theta_1}w_1, e^{i\theta_1}w_2, e^{i(\theta_1+\theta_2)}w_3 )$ where $\vec w$ satisfies $E(\vec w)=\inf\limits_{\vec u\in \mathcal{P}^{-}_{r,a_1,a_2}} E(\vec u)$ and $(\theta_1,\theta_2)\in \R^2$ and $\mathcal{P}^{-}_{r,a_1,a_2}=\mathcal{P}^{-}_{a_1,a_2}\cap H^1_r(\R^N,\R^3)$.

\begin{Lem}\label{lem3.1}
Suppose that $\max\{a_1,a_2\}<D$, $\alpha,a_1, a_2>0$, $2_*<p<2^*$ and $N\le 3$, then there exists  $\alpha_0>0$ such that
\begin{equation*}
0<m^{-}(a_1,a_2):=\inf_{\vec{u}\in \mathcal{P}^{-}_{a_1,a_2}}E(\vec{u})<\min\left\{m(a_1),m(a_2)\right\},
\end{equation*}
for any $\alpha >\alpha_0$, where $m(a_1)$ and $m(a_2)$ are defined in \eqref{c1}.
\end{Lem}
\begin{proof}
Using point (i) of  Lemma \ref{lem2.11}, we can take $\rho_{\max}>0$ such that $\max_{\rho\in \R^{+}}h(\rho)=h(\rho_{\max})>0$. Therefore, there exists a strictly positive $\sigma_{\vec{u}}=\frac{ \rho_{\max} }{ \left(\sum^3_{i=1}\|\nabla u_i\|^2_2\right)^{\frac{1}{2}} }$ such that $\left(\sum^3_{i=1}\|\nabla (\sigma_{\vec{u}}\star u_i)\|^2_2\right)^{\frac{1}{2}}=\rho_{\max}$ for any $\vec{u} \in\mathcal{P}^-_{a_1,a_2}$. Points (ii), (iii) of Lemma \ref{lem2.3} indicate that $s=0$ is the unique strict maximum of $\Psi_{\vec{u}}(s)$, so we have
\begin{equation*}
E(\vec{u})=\Psi_{\vec{u}}(0)\ge \Psi_{\vec{u}}(\sigma_{\vec{u}})= E(\sigma_{\vec{u}}\star \vec{u})\ge h\left(\Big(\sum^3_{i=1}\|\nabla (\sigma_{\vec{u}}\star u_i)\|^2_2\Big)^{\frac{1}{2}}\right)=h(\rho_{\max})>0.
\end{equation*}
Consequently, we obtain $\inf_{\vec{u}\in \mathcal{P}^-_{a_1,a_2}}E(\vec{u})\ge \max_{\rho\in \R^{+}}h(\rho)>0$.
For fixed $a_1,a_2>0$, by Lemma \ref{lem2.12}, $m(b)$ is achieved by $u_0\in S(b)\cap H^1(\R^N,\R^3)$ for any $0<b$. Let $u_1$ be the positive solution of \eqref{b1} with parameter $\|u_1\|^2_2=b^2$, $u_2$ be the positive solution of \eqref{b1} with $\|u_2\|^2_2=a^2_1-b^2$ and $u_3$ be the positive solution of \eqref{b1} with $\|u_3\|^2_2=a^2_2-b^2$. We have $(u_1,u_2,u_3)\in S(a_1,a_2)$, and it is easy to see that
\begin{equation*}
J(s\star u_1)\to 0, \quad J(s\star u_2)\to 0, \quad J(s\star u_3)\to 0,\quad \text{as}\quad s\to 0^+,
\end{equation*}
see \eqref{z2} for the definition of $J$.
Therefore, there exists $s_0>0$ small enough which is independent of $\alpha$ such that
\begin{equation*}
\begin{aligned}
\max_{0<s<s_0}E(s\star(u_1,u_2,u_3))&<\max_{0<s<s_0}J(s\star u_1)+J(s\star u_2)+J(s\star u_3)\\
&<\min\left\{m(a_1),m(a_2)\right\},
\end{aligned}
\end{equation*}	
as both $m(a_1)$ and $m(a_2)$ are strictly positive.
If $s\ge s_0$, then the interaction term is bounded from below as in the following:
\begin{equation*}
\int (s\star u_1)(s\star u_2)(s\star u_3)=s^{\frac{N}{2}}\int u_1u_2u_3\ge Ks_0^{\frac N2},
\end{equation*}
where $K>0$. Thus, we have
\begin{equation*}
\begin{aligned}
	\max_{s\ge s_0}E\left(s\star(u_1,u_2,u_3)\right)&\le\max_{s\ge s_0}J(s\star u_1)+J(s\star u_2)+J(s\star u_3)-\alpha Ks_0^{\frac N2} \\
		&\le m(b)+m\left(\sqrt{a^2_1-b^2}\right)+m\left(\sqrt{a^2_2-b^2}\right)-\alpha Ks_0^{\frac N2}.
\end{aligned}
\end{equation*}
From Lemma \ref{lem2.12}, $m(b)$ is strictly decreasing for $b>0$, then $m(b)\ge \max\{m(a_1),m(a_2)\}$.
It is clear that there exists $\alpha_0>0$ such that
\begin{equation*}
\max\limits_{s\ge s_0}E(s\star(u_1,u_2,u_3))<\min\left\{m(a_1),m(a_2)\right\}      \;\; \hbox{ for all }  \;\; \alpha>\alpha_0.
\end{equation*}
Hence, the proof is complete.
\end{proof}

\begin{Lem}\label{lem3.3}
Let $\max\{a_1,a_2\}<D$, $a_1, a_2>0$, $2_*<p<2^*$ and $N\le 3$. There exists $\alpha_0>0$ such that for all $\alpha>\alpha_0$, $m^-(a_1,a_2)$ is achieved by some $\vec{v}$, which is real-valued, positive, radially symmetric and decreasing.
\end{Lem}
\begin{proof}
We only need to show that $m^-_r(a_1,a_2)$ is attained. If $N=2,3$,   we refer to \cite{Soave1,JL,BJS,CZZ} for the existence of a radially symmetric Palais-Smale sequence $\{\vec u_n\}$ at the level $m^-_r(a_1,a_2)$ and $P(\vec{u}_n)\to 0$. If $N=1$, combine \cite[Remark 5.2]{Soave1} with \cite[Lemma 3.1]{CZZ} with the necessary modifications.
Therefore, we can choose a non-negative and radially symmetric Palais-Smale sequence $\{\vec{u}_n\}$ for $m^-_r(a_1,a_2)$ with $P(\vec{u}_n)=o_n(1)$, that is $\lim\limits_{n\to \infty} E(\vec{u}_n)=m^-_r(a_1,a_2)$ and $E'|_{S(a_1,a_2)}\to 0$ as $n\to \infty$.
Similarly to the proof of Lemma \ref{lem2.6}, we have that sequence $\{\vec{u}_n\}$ is bounded in $H^1(\R^N,\mathbb{C}^3)$, and there exists $(u_1,u_2,u_3)$ such that $(u_{1,n},u_{2,n},u_{3,n})\rightharpoonup (u_1,u_2,u_3)$ in $H^1(\R^N,\mathbb{C}^3)$. Hence, $u_i\ge 0$ are radial functions for all $i=1,2,3$.

\vskip1mm
\noindent We claim that $u_1\not\equiv 0$, $u_2\not\equiv 0$, and $u_3\not\equiv0$.
\vskip1mm
As in Lemma \ref{lem2.6}, we can separate the analysis in three cases and the proof is similar except for second case. In this case, by the maximum principle and Brezis-Lieb Lemma (cf. \cite{BL, hanq}), we derive a contradiction to $m^-_r(a_1,a_2)<\min\{m(a_1),m(a_2)\}$ as $\alpha>\alpha_0$.

So we can apply a similar argument as the proof of Lemma \ref{lem2.6}.
Therefore, we then conclude that $\vec{u}_n\to \vec{u}$ in $H^1_r(\R^N,\R^3)$ and $E(\vec{u})=m^-_r(a_1,a_2)$.
\end{proof}

At this point, we study the semi-trivial limit behavior as $a_1> 0$ and $a_2\to 0$.
\begin{Lem}\label{lem4.4}
Let  $\alpha,a_1,a_2>0$, and $2_*<p<2^*$ for $N\le 3$. If $a_1\neq 0$ is fixed and  $a_2\to 0$ (or similarly if $a_1\to 0$ and $a_2\neq 0$ is fixed), then for the second solution $\vec{v}$ of \eqref{eqA0.2}, up to a subsequence,
we have $m^-(a_1,a_2)\to m(a_1)$, and
\begin{equation*}
\left(\tilde{\kappa}^{-\frac{1}{p-2}}v_1\left(\tilde{\kappa}^{-\frac{1}{2}}x\right),v_2(x),v_3(x)\right)\to (w_p,0, 0) \ \  \text{in}\ H^{1}(\R^N,\mathbb{C}^3),
\end{equation*}
where $\tilde{\kappa}=\left(\frac{a^2_1}{\|w_p\|^2_2} \right)^{\frac{p-2}{2-p\gamma_p}}$ and $w_p$ is the positive radial solution of $-\Delta w+w=|w|^{p-2}w$.
\end{Lem}
\begin{proof}
An analysis similar to that in the proof of \cite[Lemma 2.6]{KO} show that for $a_1,a_2\ge 0$, $m^-(a_1,a_2)$ is continuous at $(a_1,a_2)$. By Theorem \ref{th1.2}, for $a_{1,n},a_{2,n}>0$, there exists $(u_{1,n},u_{2,n},u_{3,n})\in H^1_r(\R^N,\mathbb{C}^3)\cap S(a_{1,n},a_{2,n})$ such that
\[
P(\vec{u}_{n})=o_n(1) \quad \hbox{and} \quad E(\vec{u}_{n})\to m^-(a_{1,n},a_{2,n})
\]
provided $\alpha$ is large enough. We assume that $a_{1,n}\to a_1$ and $a_{2,n}\to 0$. Then we have that $\|u_{1,n}\|^2_2\to a^2_1$, $\|u_{2,n}\|^2_2\to 0$ and $\|u_{3,n}\|^2_2\to 0$, and
$\vec{u}_n$ is a bounded sequence in $H^1(\R^N,\R^3)$. There exists $u_1\in H^1(\R^N,\R)$ such that $u_{1,n}\rightharpoonup u_1$ and $u_{2,n}\rightharpoonup 0$ and $u_{3,n}\rightharpoonup 0$. Therefore, we have $\int u_{1,n}u_{2,n}u_{3,n}\to 0$. Moreover, by the Lagrange multipliers rule there exists $\omega_n\in \R$ such that
\begin{equation*}
\int  (\nabla u_{1,n}\nabla\phi-\lambda_n u_{1,n}\phi-|u_{1,n}|^{p-2}u_{1,n}\phi)=o_n(1)\|\phi\|_{H^1(\R^N)},
\end{equation*}
for all $\phi \in H^1(\R^N,\R)$. The choice $\phi=u_{1,n}$ gives
\begin{equation*}
\lambda_n a^2_1=\|\nabla u_{1,n}\|^2_2-\|u_{1,n}\|^p_p+o_n(1).
\end{equation*}
Moreover,  the boundedness of $\{\vec{u}_n\}$ in $H^1(\R^N,\R^3)$ implies that $\{\lambda_n\}$ is bounded as well, thus $\lambda_n\to \lambda_1\in \R$.
Similarly, since $\|u_{2,n}\|^2_2\to 0$ and $\|u_{3,n}\|^2_2\to 0$, we have $u_{2,n},u_{3,n}\to 0$ in $H^1(\R^N,\R)$. Recalling that  $P(\vec{u}_n)\to 0$,
\begin{equation*}
o_n(1)=P(\vec{u}_{n})=\sum^{3}_{i=1}\|\nabla u_{i,n}\|^2_2-\gamma_p\|u_{1,n}\|^p_p+o_n(1)
=\|u_{1,n}\|^2_2-\gamma_p\|u_{1,n}\|^p_p+o_n(1),
\end{equation*}
we have
\begin{equation*}
\lambda_n a^2_1=\left(1-\gamma_p\right)\|u_{1,n}\|^p_p+o_n(1).
\end{equation*}
Since $\gamma_p<1$, we deduce that $\lambda_1\ge 0$ with equality only if $u_1\equiv 0$. But $u_1$ cannot be identically 0 because $E(\vec{u}_n)\not\rightarrow 0$. Then, up to a subsequence, $\lambda_n\to \lambda_1>0$.
By weak convergence, $u_1$ is a radial weak  solution of $-\Delta u+\lambda_1 u=|u|^{p-2}u$. We infer that
\begin{equation*}
\int (|\nabla (u_{i,n}-u_1)|^2+\lambda_1|u_{1,n}-u_1|^2)=o_n(1),
\end{equation*}
and  $u_{1,n}\to u_1$ in $H^1_r(\R^N,\R)$. In addition,
\begin{equation*}
\begin{aligned}
E(\vec{u}_n)&=\frac{1}{2}\|u_{1,n}\|^2_2-\frac{1}{p}\|u_{1,n}\|^p_p+o_n(1)=\frac{1}{2}\left(1-\frac{1}{\gamma_p}\right)\|u_{1,n}\|^p_p+o_n(1)\\
&=\frac{1}{2}\left(1-\frac{1}{\gamma_p}\right)\|u_{1}\|^p_p+o_n(1)=m(a_1)+o_n(1).
\end{aligned}
\end{equation*}
By rescaling, $u_1=\tilde{\kappa}^{\frac{1}{p-2}}w_p(\tilde{\kappa}^{\frac{1}{2}}x)$ where $\tilde{\kappa}=\left(\frac{a^2_1}{\|w_p\|^2_2} \right)^{\frac{p-2}{2-p\gamma_p}}$ and $w_p$ is the positive radial solution of $-\Delta w+w=|w|^{p-2}w$.
\end{proof}

\noindent \textbf{Proof of Theorem \ref{th1.2}.}
\textup{(i)} It follows from Lemmas \ref{lem2.5} and \ref{lem3.3} that there is a mountain-pass critical point of $E$ on $S(a_1,a_2)$.
Therefore, there exists $\vec{v}\in S(a_1,a_2)$ such that $E(\vec{v})=m^-(a_1,a_2)$. \smallskip

\noindent \textup{(ii)} It follows from  Lemma \ref{lem4.4}.

\vskip1mm
\section{Proof of Theorem \ref{th1.4}}
In this section we prove the global existence result. We observe that the following identity holds true:
\begin{equation*}\label{ident}
E(\vec{\psi})-\frac{1}{p\gamma_p}P(\vec{\psi})=\frac{1}{2}\left(1-\frac{2}{p\gamma_p}\right)\sum^{3}_{i=1}\|\nabla \psi_i\|^2_2-\alpha\left(1-\frac{1}{p-2}\right)\mathrm{Re}\int \psi_1\psi_2\overline{\psi}_3,
\end{equation*}
recalling the definition of the energy \eqref{eq:energy} and the Pohozaev functional in \eqref{poho-intro}.\\

\noindent \textbf{Proof of Theorem \ref{th1.4}.}
From \cite[Chapter 4]{CT}, we get that \eqref{eqA0.1} is locally well-posed, therefore, $\vec{\psi}\in  C\left([0,T_{\max}), H^1(\R^N,\mathbb{C}^3)\right)$ for some $T_{\max}>0$, and by the blow-up alternative $T_{\max}=+\infty$ or $\sum^{3}_{i=1}\|\nabla \psi_i(t)\|^2_2\to +\infty$ as $t\to T^{-}_{\max}$. We assume by contradiction that $\sum^{3}_{i=1}\|\nabla \psi_i(t)\|^2_2\to +\infty$ as $t\to T^{-}_{\max}$. We omit the time dependence when no confusion may arise. By the Gagliardo-Nirenberg inequality,
\begin{equation*}
\begin{aligned}
&E(\vec{\psi})-\frac{1}{p\gamma_p}P(\vec{\psi})\\
&\quad =\frac{1}{2}\left(1-\frac{2}{p\gamma_p}\right)\sum^{3}_{i=1}\|\nabla \psi_i\|^2_2-\alpha\left(1-\frac{1}{p-2}\right)\mathrm{Re}\int \psi_1\psi_2\overline{\psi}_3\\
&\quad \ge \frac{1}{2}\left(1-\frac{2}{p\gamma_p}\right)\sum^{3}_{i=1}\|\nabla \psi_i\|^2_2
-\frac{\alpha C^{3}(N,p)}{3}\max\left\{a^{\frac{6-N}{2}}_1,a^{\frac{6-N}{2}}_2\right\}\left(\sum^{3}_{i=1}\|\nabla \psi_i\|^2_2\right)^{\frac{N}{4}}.\\
\end{aligned}
\end{equation*}
Therefore, we have
\begin{equation*}
E(\vec{\psi}(t))-\frac{1}{p\gamma_p}P(\vec{\psi}(t))\to +\infty \quad \text{as} \quad t\to T^{-}_{\max},
\end{equation*}
and by conservation of the energy, it follows that $P(\vec{\psi}(t))\to -\infty$ as $t\to T^{-}_{\max}$.

\vskip1mm

\noindent We claim, with a strategy as in \cite{Soave1}, that there exists $K>0$ such that $t_{\vec{\psi}_0}<1$ for all $\vec{\psi}_0\in S(a_1,a_2)$ with $P(\vec{\psi}_0)<-K$. \\
We separate two cases.
At first, suppose that $\vec{\psi}_0\in \mathcal{M}$, then by the Gagliardo-Nirenberg inequality,
\begin{equation*}
\begin{aligned}
P(\vec{\psi}_0)&\ge\sum^{3}_{i=1}\|\nabla \psi_{0,i}\|^2_2-\gamma_pC^{p}(N,p)\max\left\{a^{\frac{p-p\gamma_p}{2}}_1,a^{\frac{p-p\gamma_p}{2}}_2\right\}\left(\sum^{3}_{i=1}\|\nabla \psi_{0,i}\|^2_2\right)^{\frac{p\gamma_p}{2}}\\
&\quad -\frac{N\alpha C^{3}(N,p)}{2}\max\left\{a^{\frac{6-N}{2}}_1,a^{\frac{6-N}{2}}_2\right\}\left(\sum^{3}_{i=1}\|\nabla \psi_{0,i}\|^2_2\right)^{\frac{N}{4}}.\\
\end{aligned}
\end{equation*}
This implies that $P(\vec{\psi}_0)\ge g\left(\left(\sum^{3}\limits_{i=1}\|\nabla \psi_{0,i}\|^2_2\right)^{\frac{1}{2}}\right)$, where
\begin{equation*}
g(y)=y^2-\gamma_pC^{p}(N,p)\max\left\{a^{\frac{p-p\gamma_p}{2}}_1,a^{\frac{p-p\gamma_p}{2}}_2\right\}y^{p\gamma_p}-\frac{N\alpha C^{3}(N,p)}{2}\max\left\{a^{\frac{6-N}{2}}_1,a^{\frac{6-N}{2}}_2\right\}y^{\frac{N}{2}}.
\end{equation*}
As in the proof of Lemma \ref{lem2.11}, under the assumption of $\max\{a_1,a_2\}<D$, there exists $R_2,R_3>0$ such that $g$ is positive on $(R_2,R_3)$. Since $\lim_{y\to 0^+}g(y)=0^-$ and $g$ is continuous, there exists $K>0$ such that $g(y)\ge-K$ on $[0,R_2]$. From Lemma \ref{lem2.3},   we get that $s_{\vec{\psi}_0}$ is the local minimizer of $\Psi_{\vec{\psi}_0}$, and hence
\begin{align*}
\inf_{s\in (0,s_{\vec{\psi}_0})}s\Psi'_{\vec{\psi}_0}(s)&=\inf_{s\in (0,s_{\vec{\psi}_0})}P(s\star\vec{\psi}_0)\\
&\ge \inf_{s\in(0,s_{\vec{\psi}_0})} g\left(s\star\left(\sum^{3}_{i=1}\|\nabla \psi_i\|^2_2\right)^{\frac{1}{2}}\right)\ge \inf_{y\in (0,R_2)}g(y)\ge -K.
\end{align*}
We assume by contradiction that $P(\vec{\psi}_0)<-K$ but $t_{\vec{\psi}_0}\ge 1$. If $1\in [s_{\vec{\psi}_0},t_{\vec{\psi}_0}]$, then we have $P(\vec{\psi}_0)=\Psi'_{\vec{\psi}_0}(1)\ge 0$, which is impossible. If $s_{\vec{\psi}_0}>1$, it follows that
\begin{equation*}
-K>P(\vec{\psi}_0)=\Psi'_{\vec{\psi}_0}(1)\ge \inf_{s\in (0,s_{\vec{\psi}_0})} s \Psi'_{\vec{\psi}_0}(s) \ge -K,
\end{equation*}
which is a contradiction.
\vskip1mm
\noindent Secondly, suppose that $\vec{\psi}_0\not\in \mathcal{M}$, let $t_{\vec{\psi}_0}$ be the unique critical point of the function $\Psi_{\vec{u}}$ which  is a strict maximum point at positive level. Then  $t_{\vec{\psi}_0}<1$ for $\vec{\psi}_0\in S(a_1,a_2)$ with $P(\vec{\psi}_0)<-K$. Thus, the proof of  the claim is complete.

\vskip1mm
\noindent Since $P\left(\vec{\psi}(t)\right)\to -\infty$ as $t\to T^{-}_{\max}$, by the above claim and Lemma \ref{lem2.3}, it gives that $t_{\vec{\psi}(T_{\max}-\varepsilon)}<1$ if $\varepsilon$ is small enough. It follows from $P(\vec{\psi}_0)>0$ that $t_{\vec{\psi}_0}>1$, and since $\vec{\psi}_0\mapsto t_{\vec{\psi}_0}$ is continuous in $H^1(\R^N,\mathbb{C}^3)$, then there exists $\tau\in (0,T_{\max})$ such that $t_{\vec{\psi}(\tau)}=1$, i.e., $\vec \psi(\tau)\in \mathcal{P}^-_{a_1,a_2}$. The conservation of the energy and  the assumption on $E(\vec{\psi}_0)$ yields
\begin{equation*}
\inf_{\vec u\in \mathcal{P}^-_{a_1,a_2}}E(\vec u)>E(\vec{\psi}_0)=E(\vec\psi(\tau))\ge \inf_{\vec u\in \mathcal{P}^-_{a_1,a_2}}E(\vec u),
\end{equation*}
which is a contradiction.

\section{Proof of Theorem \ref{thm:blowup}}

In this last section, we prove that the conditions in Theorem \ref{thm:blowup} are sufficient to have formation of singularities in finite time, as well as the instability result. 
\begin{Lem}\label{lem5.4}
Under the assumption of Theorem \ref{th1.2}, let $\vec{\psi}(t)$ be the solution of \eqref{eqA0.1} with initial datum $\vec{\psi}_0 \in S(a_1,a_2)$, $P(\vec{\psi}_0)<0$ and $E(\vec{\psi}_0)<\inf E(\vec v)$. Then there exists $\eta>0$ such that $P(\vec \psi (t))\leq -\eta<0$ for any $t$ in the maximal time of existence.
\end{Lem}
\begin{proof}
Similar to the proof of Lemma \ref{lem2.11}, $t_{\vec{\psi}_0}$ is the unique global maximal point of $\Psi_{\vec{\psi}_0}$, and $\Psi_{\vec{\psi}_0}$ is strictly decreasing and concave in $(t_{\vec{\psi}_0},+\infty)$,  see \eqref{eq:fiber}  for the definition of $\Psi_{\vec \psi_0}$. From \cite[Section 10]{Soave1},
we have the following claim, if $\vec{\psi}_0\in S(a_1,a_2)$ and $t_{\vec{\psi}_0}\in (0,1)$, then
\begin{equation}\label{x5}
P(\vec{\psi}_0)\le E(\vec{\psi}_0)-\inf\limits_{\vec u\in \mathcal{P}^-_{a_1,a_2}}E(\vec u).
\end{equation}
Let $\vec{\psi}(t)$ be the solution of \eqref{eqA0.1} with initial datum $\vec{\psi}(0):=\vec{\psi}_0$, defined on the interval $[0, T_{\max})$. By continuity, and $P(\vec{\psi}_0)<0$, provided $t$ is sufficiently small we have $P(\vec{\psi}(t))<0$. Therefore, from \eqref{x5},
\begin{equation}\label{P-upper-bound}
P(\vec{\psi}(t))\le  E(\vec{\psi}(t))-\inf\limits_{\vec u\in \mathcal{P}^-_{a_1,a_2}}E(\vec u)=E(\vec{\psi}_0)-\inf\limits_{\vec u\in \mathcal{P}^-_{a_1,a_2}}E(\vec u)=:-\eta<0,
\end{equation}
for any $t$. Hence, we deduce from the continuity that $P(\vec{\psi}(t))<-\eta$ for all $t\in [0,T_{\max})$.
\end{proof}
 The next result is a refinement of the Lemma \ref{lem5.4}.
\begin{Lem}\label{lem6}
Under the same hypothesis of Lemma \ref{lem5.4}, there exists a positive constant $\delta>0$ such that
\[
P(\vec\psi(t))\leq -\delta\sum_{i=1}^3\|\nabla \psi_i(t)\|_{2}^2.
\]
\end{Lem}
\begin{proof}
From the proof of Lemma \ref{lem5.4} we already know that there exists a positive $\eta>0$ such that $P(\vec\psi(t))\leq -\eta$ in the maximal time of existence of the solution, see \eqref{P-upper-bound}. By the algebraic relation (we omit the time dependence on $\psi_i$)
\[
E(\vec\psi)-\frac{1}{p\gamma_p}P(\vec\psi)=\frac{1}{2}\left(1-\frac{2}{p\gamma_p}\right)\sum^{3}_{i=1}\|\nabla \psi_i\|^2_2-\alpha\left(1-\frac{1}{p-2}\right)\mathrm{Re}\int \psi_1\psi_2\overline{\psi}_3,
\]
we have that
\[
\sum^{3}_{i=1}\|\nabla \psi_i\|^2_2=\frac{2p\gamma_p}{p\gamma_p-2}\left( E(\vec\psi)-\frac{1}{p\gamma_p}P(\vec\psi)+\frac{\alpha(p-3)}{p-2}\mathrm{Re}\int \psi_1\psi_2\overline{\psi}_3 \right).
\]
Therefore,
\begin{equation}\label{PGE-ide}
P(\vec\psi)+\delta\|\nabla \vec\psi\|_{2}^2=\left( 1 -\frac{2\delta}{p\gamma_p-2} \right)P(\vec \psi) +\frac{2\delta p\gamma_p}{p\gamma_p-2}E(\vec\psi)+\frac{2\alpha\delta p\gamma_p(p-3)}{(p\gamma_p-2)(p-2)}\mathrm{Re}\int \psi_1\psi_2\overline{\psi}_3.
\end{equation}
By the H\"older and the Gagliardo-Nirenberg  inequalities, jointly with the conservation of the masses, see \eqref{eq:cons-masses},
\[
\int \psi_1\psi_2\overline{\psi}_3\leq\|\psi_1\|_{3}\|\psi_2\|_{3}\|\psi_3\|_{3}\lesssim \left(\|\nabla \psi_1\|_{2} \|\nabla \psi_2\|_{2}\|\nabla \psi_3\|_{2}\right)^{N/6}\lesssim \left(\sum_{i=1}^3\|\nabla \psi_i\|_{2}^2\right)^{N/4}.
\]
For $N=2,3$, $N/4<1$, and hence by the generalized Young's inequality
\[
\frac{2\delta\alpha\delta p\gamma_p(p-3)}{(p\gamma_p-2)(p-2)}\mathrm{Re}\int \psi_1\psi_2\overline{\psi}_3\leq \frac\delta2\sum_{i=1}^3\|\nabla \psi_i\|_{2}^2+C\delta.
\]
By inserting the above estimate in \eqref{PGE-ide}, and using the conservation of the energy, we get
\[
P(\vec\psi(t))+\frac\delta2\sum_{i=1}^3\|\nabla \psi_i(t)\|_{2}^2\leq-\left( 1 -\frac{2\delta}{p\gamma_p-2} \right)\sigma+\delta C,
\]
so by choosing $\delta$ sufficiently small we obtain the desired result, as the right-hand side can be made strictly negative uniformly in time.
\end{proof}

\subsection{Proof of Theorem \ref{thm:blowup}}We can now prove the blow-up results.
Define
\begin{equation}\label{eq:virdef}
I(t)=\sum_{i=1}^3\int \varphi|\psi_i(t)|^2\,dx
\end{equation}
for a smooth, real, non-negative,  time independent function $\varphi=\varphi(x)$. By differentiating twice in time and using \eqref{eqA0.1}, we get (we omit the time dependence on $\psi_i$)
\begin{equation}\label{eq:vir-2}
\begin{aligned}
I^{\prime\prime}(t)&= \sum_{i=1}^3 \left(4\mathrm{Re}\left\{\int\nabla^2\varphi  \nabla \psi_i\nabla\overline \psi_i\right\}-\int\Delta^2 \varphi |\psi_i|^2-2\left(1-\frac2p\right)\int\Delta\varphi |\psi_i|^p\right)\\
&-2\alpha\mathrm{Re}\int\Delta\varphi \psi_1\psi_2\overline\psi_3.
\end{aligned}
\end{equation}
By plugging $\varphi=|x|^2$ in \eqref{eq:virdef}, and using \eqref{eq:vir-2} along with Lemma \ref{lem5.4}, after integrating in time twice we obtain
\begin{equation*}
0\le I(t) \le - 8\eta t^2+O(t) \qquad\forall\, t\in [0,T_{\max}),
\end{equation*}
and a convexity argument gives $T_{\max}<\infty$.
\medskip

We now consider  radial solutions. Let $\chi: [0,\infty) \rightarrow [0,\infty)$ be a  $C^\infty$, non-negative function satisfying
 \begin{align*} 
 \chi(r):= \left\{
 \begin{array}{ccl}
 r^2 &\text{if}& 0\leq r \leq 1, \\
 \text{const.} &\text{if}& r\geq 2,
 \end{array}
 \right.
 \quad \chi''(r) \leq 2, \quad \forall\, r \geq 0.
 \end{align*}
 Given $R>1$, we define by rescaling, the radial function $\varphi_R: \R^N \rightarrow \R$ by
 \begin{align*} 
  \varphi_R(x) = \varphi_R(r) := R^2 \chi(r/R).
 \end{align*}
If $\varphi$ is radial and $\vec \psi$ is also radial, then
\begin{align*}
I^{\prime\prime}(t)&= \sum_{i=1}^3\left( 4\int \varphi''_R(r) |\nabla\psi_i|^2 -\int \Delta^2 \varphi_R|\psi_i|^2   -2\left(1-\frac2p\right)\int\Delta\varphi_R |\psi_i|^p\right) \\
 & -2\alpha\mathrm{Re} \int\Delta\varphi_R \psi_1\psi_2\overline\psi_3\\
&= 8\sum_{i=1}^3 \int |\nabla\psi_i|^2 +4\sum_{i=1}^3 \int (\varphi''_R(r)-2) |\nabla\psi_i|^2 -\sum_{i=1}^3\int \Delta^2 \varphi_R|\psi_i|^2  \\
&+2\left(1-\frac2p\right)\sum_{i=1}^3 \int(2N-\Delta\varphi_R )|\psi_i|^p-4N\left(1-\frac2p\right)\sum_{i=1}^3 \int |\psi_i|^p \\
&+2\alpha\mathrm{Re} \int(2N-\Delta\varphi_R ) \psi_1\psi_2\overline\psi_3-4N\alpha2\mathrm{Re}\int \psi_1\psi_2\overline\psi_3.
\end{align*}
By using the properties of the localization function $\varphi_R$, and the conservation of masses (namely, the quantities $Q_1$ and $Q_2$, see \eqref{eq:cons-masses}), we estimate
\begin{equation}\label{eq:vir1}
\begin{aligned}
I^{\prime\prime}(t)&\leq 8\sum_{i=1}^3 \int |\nabla\psi_i|^2 +CR^{-2}-4N \left(1-\frac2p\right)\sum_{i=1}^3\int |\psi_i|^p-8N\alpha\mathrm{Re}\int \psi_1\psi_2\overline\psi_3\\
&+C\sum_{i=1}^3 \int_{|x|\geq R}|\psi_i|^p+2\alpha\int_{|x|\geq R}|\psi_1\psi_2\overline\psi_3| \\
&=8P(\vec \psi)+CR^{-2}+C\sum_{i=1}^3 \int_{|x|\geq R}|\psi_i|^p+2\alpha\int_{|x|\geq R}|\psi_1\psi_2\overline\psi_3|.
\end{aligned}
\end{equation}

\noindent To estimate the last term, we recall the following radial Sobolev embedding (see e.g. \cite{COza}): for a radial function $f\in H^1(\R^N)$, we have  for $\frac12\leq s<1$ and $N\geq 2$,
\begin{align} \label{rad-sobo}
\sup_{x \ne 0} |x|^{\frac N2-s} |f(x)| \leq C\|\nabla f\|^{s}_{2} \|f\|^{1-s}_{2}.
\end{align}
Thanks to \eqref{rad-sobo} and the conservation of mass, we estimate with $s=\frac12$,
\begin{equation}\label{eq:strauss}
\begin{aligned}
\int_{|x|\geq R} |\psi_i|^{p} &=\int_{|x|\geq R} |\psi_i|^2|\psi_i|^{p-2}\lesssim \left( R^{-\frac{(N-1)}{2}}\|\nabla \psi_i\|_{2}^{1/2}\|\psi_i\|_{2}^{1/2}\right)^{p-2} \|\psi_i\|_{2}^2\\
&\lesssim R^{-\frac{(N-1)(p-2)}{2}}\|\nabla \psi_i\|_{2}^{(p-2)/2}.
\end{aligned}
\end{equation}
By H\"older and Cauchy-Schwarz inequalities, and by  \eqref{eq:strauss} with $p=3$ we get
\begin{equation}\label{strauss-p3}
\int_{|x|\geq R}|\psi_1\psi_2\overline\psi_3| \lesssim R^{-\frac{N-1}{2}} \sum_{i=1}^3 \|\nabla \psi_i\|_{2}^{1/2}.
\end{equation}
Hence, from \eqref{eq:vir1}, \eqref{eq:strauss}, and \eqref{strauss-p3} we get
\begin{equation}\label{eq:vir-est}
I^{\prime\prime}(t)\leq8P(\vec\psi)+CR^{-2} +C R^{-\frac{(N-1)(p-2)}{2}}\sum_{i=1}^3 \|\nabla \psi_i\|_{2}^{(p-2)/2}+ R^{-\frac{N-1}{2}} \sum_{i=1}^3 \|\nabla \psi_i\|_{2}^{1/2}.
\end{equation}
Let us observe that in dimension $N=3$ it holds true that  $\frac{p-2}{2}<2$ provided $p<6=2+\frac{4}{N-2}=p^*$, which fits our assumption in the three-dimensional setting. When $N=2$, we must restrict the range of the non-linearity to $p\in(4,6)$. See also Ogawa and Tsutsumi \cite{OgTs}.
\vskip2mm
A convexity argument yields the blow-up result, by glueing together \eqref{eq:vir-est},  \eqref{P-upper-bound} and Lemma \ref{lem6}, provided $R$ is large enough.

\begin{Rem}\rm
In the three-dimensional case, the radial symmetry can be further relaxed to a cylindrical symmetric setting, provided we impose  partial weighted $L^2$-summability of the initial data, see the first author's results in \cite{ADF, DF, BF-CV, BFG, For}.
\end{Rem}

\subsection{Proof of Corollary \ref{cor:insta} }
Let $\vec v$  be  the excited state  constructed in Theorem \ref{th1.2}, point \textup{(i)}. For any $s>0$, let $\vec{v}_s:=s\star\vec{v}$, and let $\vec{\psi}_s$ be the solution to \eqref{eqA0.1} with the initial datum $\vec{v}_s$.   Then, $\vec{v}_s\to u$ as $s\to 1^+$. By Lemma \ref{lem5.4}, it is sufficient to prove that $\vec{\psi}_s$ blows-up in finite time. In fact, it follows from \cite{BL1} that $\vec{v}\in H^1(\R^N,\R^3)$ decays exponentially at infinity, and hence $|x|\vec{v}\in L^2(\R^N,\R^3)$. Let $\sigma_{\vec{v}_s}$ be defined in Lemma \ref{lem2.3}, we have
\begin{equation*}
E(\vec{v}_s)=E(s\star\vec{v})<E(\sigma_{\vec{v}_s}\star \vec{v})=\inf_{\vec v\in \mathcal{P}^-_{a_1,a_2}}E(\vec v),
\end{equation*}
because $P(\vec{v}_s)<0$. The proof of Corollary \ref{cor:insta} is completed.

\subsection*{Conflict of interest}\rm
On behalf of all authors, the corresponding author states that there is no conflict of interest.  
\subsection*{Data availability statement }\rm
No data associated to this paper.
\subsection*{Acknowledgements}\rm 
The authors warmly thank the anonymous referee for valuable comments and suggestions which allowed to improve a previous version of the paper.\\
L. Forcella is member of the GNAMPA of the INdAM (Istituto Nazionale di Alta Matematica). X. Luo and T. Yang are supported by NNSF of China (Grant No. 12471103 and No. 12201564), Anhui Provincial Natural Science Foundation (No.2308085MA05), the Fundamental Research Funds for the Central Universities of China (No. JZ2025HGTG0255) and the Open Research Fund of Hubei Key Laboratory of Mathematical Sciences (Central China Normal University), Wuhan 430079, P. R. China. X.L. Yang is supported by NNSF of China (No. 12401130), the Postdoctoral Fellowship Program of CPSF (No. GZC20240405) and the  China Postdoctoral Science Foundation (No.2024M760761).

\end{document}